\documentclass[11pt]{article}
\usepackage{fullpage}
\usepackage[T1]{fontenc}
\usepackage{amssymb,amsmath,amsthm,MnSymbol,latexsym}
\usepackage{tikz,epsfig}
\usetikzlibrary{arrows}
\usepackage{subfigure,caption}
\usepackage{epstopdf}
\usepackage{url}
\usepackage{enumerate}
\usepackage{floatpag}
\usepackage[ruled,vlined]{algorithm2e}
\usepackage{authblk}

\makeatletter
\let\@@citation@@=\citation
\renewcommand{\citation}[1]{\@@citation@@{#1}%
\@for\@tempa:=#1\do{\@ifundefined{cit@\@tempa}%
  {\global\@namedef{cit@\@tempa}{}}{}}%
}
\makeatother

\usepackage{hyperref}

\makeatletter
\def\@lbibitem[#1]#2#3\par{%
  \@ifundefined{cit@#2}{}{\@skiphyperreftrue
  \H@item[%
    \ifx\Hy@raisedlink\@empty
      \hyper@anchorstart{cite.#2\@extra@b@citeb}%
        \@BIBLABEL{#1}%
      \hyper@anchorend
    \else
      \Hy@raisedlink{%
        \hyper@anchorstart{cite.#2\@extra@b@citeb}\hyper@anchorend
      }%
      \@BIBLABEL{#1}%
    \fi
    \hfill
  ]%
  \@skiphyperreffalse}%
  \if@filesw
    \begingroup
      \let\protect\noexpand
      \immediate\write\@auxout{%
        \string\bibcite{#2}{#1}%
      }%
    \endgroup
  \fi
  \ignorespaces
  \@ifundefined{cit@#2}{}{#3}}

\def\@bibitem#1#2\par{%
  \@ifundefined{cit@#1}{}{\@skiphyperreftrue\H@item\@skiphyperreffalse
  \Hy@raisedlink{%
    \hyper@anchorstart{cite.#1\@extra@b@citeb}\relax\hyper@anchorend
    }}%
  \if@filesw
    \begingroup
      \let\protect\noexpand
      \immediate\write\@auxout{%
        \string\bibcite{#1}{\the\value{\@listctr}}%
      }%
    \endgroup
  \fi
  \ignorespaces
  \@ifundefined{cit@#1}{}{#2}}
\makeatother

\makeatletter
\newcommand{\linkhere}[2]{%
	\phantomsection
	#1\def\@currentlabel{\unexpanded{#1}}\label{#2}%
}
\makeatother

\newtheorem{thm}{Theorem}[section]
\newtheorem{theorem}[thm]{Theorem}
\newtheorem{cor}[thm]{Corollary}
\newtheorem{corollary}[thm]{Corollary}

\newtheorem{lemma}[thm]{Lemma}

\newtheorem{claim}[thm]{Claim}
\newtheorem{conj}[thm]{Conjecture}

\newtheorem{problem}[thm]{Problem}
\newtheorem{obs}[thm]{Observation}

\theoremstyle{definition}
\newtheorem{defi}[thm]{Definition}
\newtheorem{definition}[thm]{Definition}
\newtheorem{remark}[thm]{Remark}

\usepackage{indentfirst}
\usepackage{xspace} 

\def\poly{\textrm{poly}}

\def\eps{\varepsilon}
\def\N{\mbox{\ensuremath{\mathbb N}}\xspace}
\def\R{\mbox{\ensuremath{\mathbb R}}\xspace}
\def\RR{\mbox{\ensuremath{\mathcal R}}\xspace}

\def\QQ{\mbox{\ensuremath{\mathcal Q}}\xspace}

\def\C{\mbox{\ensuremath{\mathcal C}}\xspace}
\def\F{\mbox{\ensuremath{\mathcal F}}\xspace}
\def\HH{\mbox{\ensuremath{\mathcal H}}\xspace} 
\def\I{\mbox{\ensuremath{\mathcal I}}\xspace}

\def\B{\mbox{\ensuremath{\mathcal B}}\xspace}

\newcommand\D{{\mathcal{D}}}

\newcommand\FF{{\mathcal{F}}}
\newcommand\Gc{{\mathcal{G}}}

\newcommand\Fcal{{\mathcal{F}}}

\newcommand\U{{\mathcal{U}}}

\mathchardef\mhyphen="2D

\newcommand\Hc{{\mathcal{H}}}

\newcommand\Ac{{\mathcal{A}}}
\newcommand\Bcal{{\mathcal{B}}}

\def\II{\mathcal{I}}

\mathcode`l="8000
\begingroup
\makeatletter
\lccode`\~=`\l
\DeclareMathSymbol{\lsb@l}{\mathalpha}{letters}{`l}
\lowercase{\gdef~{\ifnum\the\mathgroup=\m@ne \ell \else \lsb@l \fi}}%
\endgroup

\newcommand{\chim}{\chi_{\rm big}}

\begin{document}


\title{Coloring Geometric Hypergraphs: A Survey}
\author[1,2]{Gábor Damásdi}  
\author[1,2]{Balázs Keszegh} 
\author[1]{János Pach} 
\author[2,1]{D\"om\"ot\"or P\'alv\"olgyi}  
\author[1,3]{Géza Tóth}  

\affil[1]{HUN-REN Alfréd Rényi Institute of Mathematics, Budapest}
\affil[2]{ELTE Eötvös Loránd University, Budapest}
\affil[3]{Department of Computer Science and Information Theory, Budapest University of Technology and Economics}
\maketitle

\begin{abstract}
A family $\cal F$ of subsets of an underlying point set $S$ is called an \emph{$m$-fold covering} if every point of $S$ is contained in at least $m$ members of $\cal F$.
The starting point of our survey is a group of problems raised by the most senior author in 1980. Is it true that every $m$-fold covering of the $d$-dimensional space by geometric objects of a given type (balls, translates or homothets or congruent copies of a fixed convex or non-convex body, etc.) can be split into 2 coverings, provided that $m$ is sufficiently large, depending on the type of the objects?

The \emph{chromatic number} of a hypergraph is the smallest number of colors needed to color the vertices such that no edge of at least two vertices 
is monochromatic. 
Given a family of geometric objects $\mathcal{F}$ that covers a subset $S$ of the Euclidean space, 
we can associate it with a hypergraph 
whose vertex set is $\mathcal F$ and whose edges are those subsets ${\mathcal{F}'}\subset \mathcal F$ for which there exists a point $p\in S$ such that ${\mathcal F}'$ consists of precisely those elements of $\mathcal{F}$ that contain $p$.  
The question whether $\mathcal F$ can be split into 2 coverings is equivalent to asking whether the chromatic number of the hypergraph 
is equal to 2. 

There are a number of competing notions of the chromatic number that lead to deep combinatorial questions already for abstract hypergraphs. 
In this paper, we concentrate on \emph{geometrically defined} (in short, \emph{geometric}) hypergraphs, and survey many recent coloring results related to them. In particular, we study and survey 
the following problem, dual to the above covering question. 
Given a set of points $S$ in the Euclidean space 
and a family $\mathcal{F}$ of geometric objects of a fixed type, 
define a hypergraph $\HH_m$
on the point 
set $S$, whose edges are 
the subsets of $S$ that can be obtained as the intersection of $S$ with a
member of $\mathcal F$ and have at least $m$ elements.
Is it true that if $m$ is large enough, 
then the chromatic number of $\HH_m$ is equal to 2?


\end{abstract}


\textbf{Mathematics Subject Classification (MSC):} 05C15 (Coloring of graphs and hypergraphs)

\medskip

\section{Introduction}

The chromatic number is one of the oldest and most extensively studied parameters of graphs. Its computation is a classical NP-complete problem \cite{Karp1972}. The notion can be extended to hypergraphs (set systems) in many different ways. In 1908, Bernstein~\cite{Ber08} introduced the following interesting property of hypergraphs. 
We say that a hypergraph $\HH$ has \emph{property B}, or is \emph{$2$-colorable}, if one can split its vertex set into 2 parts such that neither of them contains an edge, except the one-vertex edges. This inspired Erd\H os~\cite{E63} to ask the following question: 
What is the smallest number $B=B(m)$ such that there exists an $m$-uniform hypergraph with $B$ edges that does not have property B? 
(A hypergraph is \emph{$m$-uniform} if each of its edges consists of $m$ vertices.) He proved 
    $$2^{m-1}<B(m)=O(m^2\cdot 2^m).$$ 
The best known lower bound, 
$$\Omega ({\sqrt {m/\log m}}\cdot 2^m)=B(m),$$
was established by Radhakrishnan and Srinivasan~\cite{RS}. 
\smallskip

Under some special conditions, however, very large, even infinite hypergraphs may be $2$-colorable 
(may have property B). 
This phenomenon is best illustrated by the \emph{Lov\'asz Local Lemma}~\cite{EL75}. In its simplest form, it yields that any hypergraph in which every edge is of size at least $m$ and intersects at most $\frac{2^{m-1}}3-2$ other edges, is $2$-colorable. A  constructive version of this lemma, which implies an efficient $2$-coloring algorithm, was established by Moser and Tardos~\cite{MoserT}.
This lemma suggested that many classes of ``nicely behaving'' hypergraphs may share some similarly favorable colorability properties. 
\smallskip

Let $\C$ be a family of sets in $\mathbb R^d$, and let $P\subseteq\mathbb{R}^d$.
We say that ${\C}$ is an {\em $m$-fold covering of $P$} if every point of $P$ belongs to at least $m$ members of $\C$.
A $1$-fold covering is simply called 
a {\em covering}. Clearly, the union of $m$ families, each of which is a covering is an $m$-fold covering.
The following geometric question was inspired by L.~Fejes T\'oth, and formulated by Pach in 1980.

\begin{problem}[Pach \cite{pach1980}]\label{thm:question}
	Is it true that every $m$-fold covering of the plane with unit disks splits into two coverings, provided that $m$ is sufficiently large?
\end{problem}

This problem can be easily reformulated as a hypergraph coloring question of the above type. 
Indeed, given an $m$-fold covering $\C$ of the plane with unit disks, define a hypergraph $\HH=\HH(\C)$ on the vertex set $V(\HH)=\C$ as follows. To each point $p$ of the plane, assign the hyperedge $e_p\in E(\HH)$ consisting of all disks in $\C$ that contain $p$. (Obviously, $e_p$ will coincide for many different points, but we remove the multiplicities in $\HH$.) It follows from the assumption that every edge of $\HH$ is of size at least $m$. If we assume, in addition, that no point of the plane is covered by \emph{too many} disks in $\C$, in terms of $m$, then it is not hard to give an upper bound on the maximum number of edges that intersect a fixed edge of $\HH$. Thus, we can apply the Lov\'asz Local Lemma to deduce that the hypergraph $\HH$ is 2-colorable \cite{PP}. This, in turn, is equivalent to saying that $\C$ can be split into two coverings.

But is the assumption on the maximum number of disks containing a given point necessary? After all, if a point $p$ is covered many times, then if we randomly split the elements of $\C$ into two parts, the probability that both parts will cover $p$ is very large. Therefore, it was expected that the answer to Pach's question is positive; Winkler~(page 137 of \cite{W07}, see also \cite{W09}) even conjectured that already $m\ge 4$ is sufficient.

\smallskip

At the turn of the millennium, Problem~\ref{thm:question} was raised in another context:  for {\em large scale ad hoc sensor networks}; see~Feige {et al.}~\cite{FeH00} and Buchsbaum {et al.}~\cite{B07}. In the---by now rather extensive---literature, this is usually referred to as the {\em sensor cover problem}. In its simplest version, it can be phrased as follows. Suppose that a large region $P$ is monitored by a set of sensors, each having a circular range of unit radius and each powered by a battery of unit lifetime. Suppose that every point of $P$ is within the range of at least $m$ sensors, that is, the family of ranges of the sensors, $\C$, forms an $m$-fold covering of $P$. If $\C$ can be split into $k$ coverings, $\C_1,\ldots,\C_k$, then the region can be monitored by the sensors for at least $k$ units of time. Indeed, at time $i$, we can switch on all sensors whose ranges belong to $\C_i\; (1 \le i \le k)$. Our goal is to maximize $k$, given $P$ and $\C$, in order to guarantee the longest possible service.
\smallskip

Obviously, Problem~\ref{thm:question} makes sense for many other classes of geometric objects, not only in the plane, but also in higher dimensions. For segments in the 
1-dimensional space, however, the situation is as nice as it can be: it follows from old, unpublished results of Gallai and Haj\'os that every $m$-fold covering of the line with segments splits into $m$ coverings. 

The analogue of Problem~\ref{thm:question} was also raised for translates of a fixed 
polygon $C$, instead of unit disks. For centrally symmetric, convex polygons Pach~\cite{pach1986} proved that the answer is positive. 
His theorem was later extended to all convex polygons~\cite{TT07,PT10,GV10}. 
Since a disk can be arbitrarily closely approximated by convex polygons, these results seem to support the conjecture that every sufficiently ``thick'' covering of the plane by unit disks, or any smooth convex objects, can be split into two coverings.
\smallskip

Surprisingly, it turned out~\cite{PP} that this is not the case! Today there is a huge literature of positive and negative results related to Problem~\ref{thm:question}; see also the webpage \cite{cogezoo}. We can say that the question has propelled research in the area for several decades.

This survey is organized as follows. In Section 2 we introduce the most important concepts,  
namely abstract and geometric hypergraphs, different types of colorings, some important constructions and basic properties.   
In the following sections we present the results on more and more complex geometric hypergraphs.
In Sections 3 and 4, we consider translates and homothets of a fixed polygon, respectively. In Section 5, we focus on shapes whose boundary consists of axis-parallel segments. Then, in Section 6, we discuss disks, unit disks and general convex shapes. 
In Sections 7 and 8, we consider two natural generalizations of geometric hypergraphs, namely intersection hypergraphs and ABA-free hypergraphs, which are related to pseudoline arrangements. 
Finally, we highlight some open problems and conjectures in Section \ref{sec:further}.


\section{Preliminaries: combinatorial and geometric}\label{sec:prelim}

Problems about graph colorings are among the oldest problems 
in graph theory. The most famous is the Four Color Conjecture of Guthrie from 1852. After many attempts, it was finally 
verified in 1976 by Appel and Haken \cite{AH89,RSST97}. It was the first major result with a computer assisted proof. 
Efforts to find the solution for more than a century greatly inspired the development of graph theory.
Birkhoff and Tutte introduced the chromatic polynomial and its generalization, the Tutte polynomial, which turned out to be important tools in algebraic graph 
theory \cite{EM11}. 
In 1960 Berge introduced perfect graphs and formulated the Strong Perfect Graph Conjecture, which was verified in 2006 \cite{CRST06}. 
The proof of Kneser's conjecture \cite{K55} by Lov\'asz \cite{L78} was the starting point of
the applications of topological methods.
Determining the chromatic number is one of the first 21 NP-complete problems of Karp \cite{K75}.
In this section we introduce hypergraphs, different types of colorings, geometric hypergraphs, and some other important 
concepts.


\subsection{Abstract hypergraphs, coloring variants}

A {\em hypergraph} $\HH=(V,E)$ is a collection of subsets $E$ of a base set $V$.
The elements of $V$ are called the {\em vertices} of the hypergraph, and the elements of $E$ the {\em hyperedges}, or simply the {\em edges} of the hypergraph.
Given a hypergraph \HH, we denote its vertex set by $V(\HH)$ and its edge set by $E(\HH)$; the $\HH$ is omitted when it leads to no confusion.
If all sets in $E$ are different, the hypergraph is called \emph{simple}. 
For technical reasons, we also assume that $\emptyset\notin E$ and $E\ne \emptyset$.
A hypergraph is finite if both $V$ and $E$ are finite sets.
A hypergraph is \emph{$m$-heavy} if all of its edges are of size at least $m$.
Edges with at least $m$ vertices are called $m$-heavy. 

The {\em incidence matrix} $M(\HH)$ of $\HH=(V,E)$ is a matrix whose rows and columns are indexed by $V$ and $E$, respectively, such that $M(v,e)=1$ if $v\in e$ and $M(v,e)=0$ if $v\notin e$.
The order of the rows and columns is arbitrary. 
The \emph{dual} of the hypergraph $\HH=(V,E)$ is the hypergraph $\HH^*=(E,V)$ in which the containment relation is reversed, that is, 
we write $e\in^* v$ if and only if $v\in e$.
Clearly, $M(\HH^*)=M^T(\HH)$. Observe that the dual of a simple hypergraph is not necessarily simple. 

The hypergraph $\HH'=(V',E')$ is a {\em subhypergraph} of $\HH=(V,E)$ if $V'\subset V$ and $E'\subset E$.
For a subset of the vertices $X\subset V$, we define the \emph{trace} of \HH on $X$ as $\HH[X]=(X,E\cap X)$, 
that is, $V(\HH[X])=X$ and $E(\HH[X])=\{e\cap X\mid e\in E(\HH),e\cap X\neq \emptyset\}$.
A subhypergraph of a trace is called a \emph{subconfiguration} or \emph{subpattern}.
It is easy to see that the incidence matrices of the subpatterns of \HH are exactly the submatrices of $M(\HH)$.
A family of hypergraphs \F is \emph{hereditary} if it is closed under taking subhypergraphs and traces, i.e., under taking subpatterns.
Equivalently, the family of incidence matrices is closed under taking submatrices.

\subsubsection{Hypergraph colorings}

A {\em $k$-coloring} of a hypergraph $\HH=(V,E)$ is a map from $V$ to $\{1,\ldots,k\}$.
There are several different versions and properties of hypergraph colorings and chromatic numbers.

\begin{enumerate}
	\item A coloring is {\em proper} if every edge of size at least two contains two vertices of different colors. 
 The \emph{chromatic number} of \HH, denoted by $\chi(\HH)$, is the smallest $k$ for which \HH admits a proper $k$-coloring.
 
 \item A $k$-coloring is {\em polychromatic} if every edge contains a vertex of each of the $k$ colors; note that such a coloring is only possible if \HH is $k$-heavy. If there is a polychromatic $k$-coloring, then there is also a polychromatic $k'$-coloring for any $k'\le k$.
	The \emph{polychromatic number} of \HH, denoted by $\chi_{\rm poly}(\HH)$, is the largest $k$ for which \HH admits a polychromatic $k$-coloring.
 
 \item A $k$-coloring is {\em rainbow} if in each edge, all vertices 
 have different colors; note that such a $k$-coloring is possible only if every edge has size at most $k$.
 If there is a rainbow $k$-coloring, then there is also a rainbow $k'$-coloring for any $k'\ge k$.
	The smallest $k$ for which \HH has a rainbow $k$-coloring is denoted by $\chi_{\rm rainbow}(\HH)$.
    This number is the same as the chromatic number of the graph (2-uniform hypergraph) 
    defined on $V(\HH)$, in which two vertices are connected by an edge if and only if there is an edge in \HH that contains both of them.
	

 \item A coloring is {\em conflict-free} if every edge $e$ contains a vertex whose color is ``unique'' within $e$, in the sense that its color differs from the color of any other vertex of $e$.
	Again, if \HH has a conflict-free $k$-coloring, then it also has a conflict-free $k'$-coloring for any $k'\ge k$.
    The smallest $k$ for which \HH has a conflict-free $k$-coloring is denoted by $\chi_{\rm cf}(\HH)$.

\end{enumerate}

If \HH is not finite, then all of the above parameters might be infinite. In this case, we could study their growth rates as functions of the number of the vertices for finite hypergraphs \HH from the given family. However, this is beyond the scope of the present work. For some beautiful questions of geometric flavor concerning these growth rates, consult \cite{C12,CFonline}.

By definition, we have the following trivial inequalities:
$$\chi(\HH) 
\le \chi_{\rm cf}(\HH)\le 
\chi_{\rm rainbow}(\HH).$$
Moreover, $\chi(\HH)=2$ if and only if $\chi_{\rm poly}(\HH)\ge 2$.

In the rest of this survey, we will only consider proper and polychromatic colorings.
As a rainbow (sometimes also called \emph{strong}) coloring is the same as the proper coloring of an underlying graph, quoting Berge~\cite{bergebook}, ``we shall not study the strong chromatic number for its own sake.''
For a survey on conflict-free colorings of geometric hypergraphs, see Smorodinsky \cite{surveycf}.

\smallskip
The notions of (proper) chromatic and polychromatic numbers can be naturally extended to families \F of 
hypergraphs by taking their maximum or minimum values over all members of the family, respectively.
Let 
$$\chi(\F)=\max \{\chi(\HH)\ \colon \ \HH\in\F\},$$  
$$\chi_{\rm poly}(\F)=\min \{\chi_{\rm poly}(\HH)\ \colon\  \HH\in\F\}.$$
In most cases, we assume that \F is hereditary. A typical example is all (finite) subconfigurations of an infinite hypergraph; see Section \ref{sec:geohyp}.

In some cases, \HH does not have a proper $k$-coloring, because of its non-heavy edges, as for these it is harder to satisfy the requirements.
For any family of hypergraphs \F, let $\chim(\F)$\footnote{Note that in many papers the notation $\chi_m$ was used instead. However, it was quite confusing that the subscript $m$ in $\chi_m$ is \emph{not} a variable, unlike at other places, so we decided to eradicate this notation.} denote the smallest integer $k$ such that there exists a (large enough) integer $m=m(k)$ with the property that every $m$-heavy hypergraph $\HH\in \F$ has a proper $k$-coloring, i.e., we have $\chi(\HH)\le k$.
If there is no such $k$, then set $\chim(\F)=\infty$.
\smallskip

The following definition is motivated by geometric questions where we want to decompose a covering into not only two, but $k$ coverings.
For any $k\ge 1$ and any hypergraph \HH, let $m_k(\HH)$ denote the smallest integer $m$ such that the subhypergraph of \HH consisting of all edges of size at least $m$ admits a polychromatic $k$-coloring. Accordingly, for any hereditary family of hypergraphs \F, let $m_k(\F)$ be the smallest $m$ for which every $m$-heavy hypergraph $\HH\in \F$ has a polychromatic $k$-coloring, i.e., we have $\chi_{\rm poly}(\HH)\ge k$.
If there is no such $m$, then let $m_k(\F)=\infty$.
By definition, $m_2(\F)<\infty$ holds if and only if $\chim(\F)=2$. Clearly, we have $$m_k(\F)\le m_{k+1}(\F),$$ for every hereditary family \F.
The following statement may also be true.

\begin{conj}\label{conj:mk}
	For any hereditary family of hypergraphs \F, if $m_2(\F)<\infty$, then $m_k(\F)<\infty$ for all $k\ge 2$.
\end{conj}

This attractive conjecture is known to be true only if $m_2(\F)=2$. This is a classic result of Berge~\cite{Ber72}. More precisely, Berge proved that $m_2(\F)=2$ implies that $m_k(\F)=k$ for every $k$. Moreover, he characterized all families \F with $m_2(\F)=2$. 

We mention that without the hereditary property, Conjecture \ref{conj:mk} is not true.
A simple example: Take any (non-empty) hypergraph $\HH$, and add a new vertex to its set of vertices, contained in each hyperedge.
Color the new vertex red and all the old ones blue.
This makes the new hypergraph $2$-colorable, but it is obviously not $3$-colorable if $\HH$ was not $2$-colorable.

Conjecture \ref{conj:mk} was first formulated in \cite{surveycd}. Later, it was popularized at several venues, including MathOverflow and problem sessions at meetings in Oberwolfach. Some other variants have also emerged.
Keszegh and Pálvölgyi conjectured that for every $k$ and every hereditary $\F$, we have 
\begin{equation}\label{m_kkeplet}
m_k(\F)\le (k-1)(m_2(\F)-1)+1.
\end{equation}
Equality holds for many geometrically defined families, as we will see later. However, it is not true in general. Pálvölgyi \cite{m2m3} exhibited a single counterexample:
a $5$-uniform hypergraph on $8$ vertices which does not have a polychromatic $3$-coloring, 
but all of its subconfigurations whose edges are of size at least $3$ are $2$-colorable, i.e., $m_2=3$ but $m_3=6$.
It is unclear how one could generalize  this counterexample.
It is still conceivable that for every $k$ and every hereditary family $\F$, we have 
 $$m_k(\F)\le C\cdot k\cdot m_2(\F),$$ 
for a suitable $C>0$, or at least
 $$m_k(\F)\le \poly(k, m_2(\F))$$ holds, where $\poly$ is a polynomial.
We know that the latter inequality is true in several special cases; see Theorem \ref{thm:selfcov}.

One can use polychromatic colorings to obtain upper bounds for the (proper) chromatic number of the union of hypergraphs. The naive bound for the union of two hypergraphs only gives $$\chi(\HH_1\cup\HH_2)\le \chi(\HH_1)\cdot\chi(\HH_2),$$ and this inequality is sharp.
However, if $\HH_1$ and $\HH_2$ have polychromatic colorings with more than 2 colors, then we can do better.





\begin{lemma}[Dam\'asdi-P\'alv\"olgyi \cite{DP22}]\label{lem:combine} Let $\HH_1,\dots, \HH_{k-1}$ be hypergraphs on a common vertex set $V$.  

If $\HH_1,\dots, \HH_{k-1}$ have polychromatic number at least $k$, then $\chi\left(\bigcup\limits_{i=1}^{k-1} \HH_i\right)\le k$.
\end{lemma}
\begin{proof}
	Let $c_i:V\to \{1,\ldots,k\}$ be a polychromatic $k$-coloring of $\HH_i$. 
	Choose $c(v)\in \{1,\ldots,k\}$ such that it differs from each $c_i(v)$.
	We claim that $c$ is a proper $k$-coloring of $\bigcup\limits_{i=1}^{k-1} \HH_i$.
	To prove this, it is enough to show that for every edge $e\in\HH_i$ and for every color $j\in\{1,\ldots,k\}$, there is $v\in e$ such that $c(v)\ne j$. 
	We can pick $v\in e$ such that $c_i(v)=j$.
	This finishes the proof.
\end{proof}

\begin{cor}\label{cor:combine}
	For any $k>1$ and families $\F_1,\ldots,\F_{k-1},$
    let $\F=\{\HH_1\cup\cdots\cup\HH_{k-1}~|~ \HH_i\in\F_i\text{ for all } i\}$.\\
    If $m_k(\F_1),\ldots,m_k(\F_{k-1})<\infty$, then
	$\chim(\F)\le k$.
\end{cor}

Lemma \ref{lem:combine} is sharp in the sense that for every $k$, there are $k-1$ hypergraphs such that each of them admits a polychromatic $k$-coloring, but their union is not properly $(k-1)$-colorable.
For example, take a cubic piece of the $(k-1)$-dimensional integer grid of width $k$, $V=\{(i_1,\ldots,i_{k-1})\mid 1\le i_j\le k\}$, and let a $k$-tuple $v^1,\ldots,v^{k}\in V$ be an edge of $\HH_j$ if the $j$-th coordinates of the $v^j$'s are all different, i.e., $\{v^1_j,\ldots,v^k_j\}=\{1,\ldots,k\}$.
A simple induction argument shows that $\chi\left(\bigcup\limits_{j=1}^{k-1} \HH_j\right)=k$.

\subsubsection{Related notions}

 A \emph{hitting set} or  \emph{transversal}, of the edges is a subset of the vertex set which intersects every edge in at least one vertex.
 For $c\ge 1$, a  
 hitting set is said to be $c$-\emph{shallow} if it intersects every edge in at most $c$ vertices \cite{SY, KP14}.
If we can find a $c$-shallow hitting set in an $m$-uniform hypergraph and its subpatterns, 
then we can construct a polychromatic $\lceil \frac mc\rceil$-coloring of it 
by repeating the following operation: color the vertices of the shallow hitting set with a new color, then delete these vertices and shrink each edge so that we obtain an $(m-c)$-uniform hypergraph (this shrinking is doable in many geometric settings).  This idea was first applied in \cite{SY}.
The existence of shallow hitting sets was also established for regular (abstract) hypergraphs in a recent work by Planken and Ueckerdt~\cite{planken2}.

Polychromatic colorings are closely related to \emph{$\eps$-nets}. Given a set system (hypergraph) \HH over an $n$-element base set $V$, a subset $N\subset V$ is called an \emph{$\eps$-net} if it intersects every edge $e\in \HH$ of size at least $\eps n$. 
If $m_k(\HH)=m$, then $V$ can be partitioned into $k$ parts such that each part intersects every $m$-heavy edge $e\in E$.
Setting $\eps=m/n$, each of these parts forms an $\eps$-net for $\HH$.
By the pigeonhole principle, at least one of the $k$ parts has at most $n/k=m/(\eps k)$ elements.
This implies that if $m_k(\HH)=O(k)$, then $\HH$ has an $\eps$-net of size $O(1/\eps)$.
For many important results on $\eps$-nets and related topics, see the excellent monograph of Mustafa \cite{nabilbook}. 

Hypergraph colorings can be applied to graphs by defining a hypergraph on the same vertex set whose hyperedges are the open (or closed) neighborhoods of the vertices of the graph. In this way polychromatic and conflict-free colorings of vertices of a graph were studied intensively, see, e.g., \cite{Cheilaris2008,PACH_TARDOS_2009,fekete-int, Erickson_2012}. 

Finally, without defining the exact problem and results, let us also mention that polychromatic colorings of the edges of graphs ~\cite{B11} or  
$t$-subsets of vertices  \cite{AKP18,BCMSSS25} have also been studied.

\subsection{Abstract hypergraph constructions}\label{subsec:abstract}

The clique number is an obvious lower bound for the chromatic number. Tutte \cite{D54} showed that they can be arbitrarily far from each other. 
He constructed the first family of triangle-free graphs with arbitrarily large chromatic number.
Now we have many constructions with this property \cite{SS20}. 
Erd\H os \cite{E59}
went even further, he constructed graphs with arbitrary large girth and chromatic number. His proof is a well known, 
beautiful application of the probabilistic method. 
Here we introduce some very useful hypergraphs of large chromatic number.







\subsubsection{Tree hypergraph}\label{sec:treehyp}

\begin{defi} For every rooted tree $T=(V,E)$, let $\Hc_T$ denote the hypergraph on vertex set $V$, whose hyperedges are the sets of the following two types.
 
 \begin{enumerate}
     \item Sibling hyperedges: for each vertex $v\in V$ that is not a leaf, take the set of all children of $v$.
      \item Descendant hyperedges: for each leaf $v\in V$, take the set of all vertices along the unique path from the root to $v$.
 \end{enumerate}
 \end{defi}

It is easy to see that if all hyperedges have at least two vertices, then $\Hc_T$ is not 2-colorable.
Indeed, we can either follow the color of the root down the tree until we reach a leaf or we get stuck at one of the vertices. In the first case we have found a monochromatic descendant edge, in the second case we have found a monochromatic sibling edge.  


The 
\emph{$m$-ary tree of depth $m$}, 
$T_m$, is the tree  that  has $m$ levels and each non-leaf vertex has $m$ children, so in total it has $\frac{m^m-1}{m-1}$ vertices. 
For every positive integer $m$ 
let $\Hc^2(m)=\Hc_{T_m}$. 
We call $\Hc^2(m)$ the \emph{m-ary tree hypergraph}. It is $m$-uniform and 
it is not $2$-colorable. 
The hypergraph family $\{\Hc^2(m)\ \colon\   m\in \N\}$ was used in \cite{PTT09} to show that $\chim>2$ for the family of primal hypergraphs of disks in the plane.

\subsubsection{Iterated tree hypergraph}\label{sec:ittree}

For constructing hypergraph families that have even greater $\chim$, 
we generalize $m$-ary trees, as follows  \cite{abab,DP20}. 

\begin{defi} Suppose we have a hypergraph $\Ac$, a hypergraph $\Bcal$, and let $F$ be an edge of $\Ac$. We define the hypergraph \emph{$\Ac$ extended by $\Bcal$ through $F$} as follows. We construct the new hypergraph starting from $\Ac$, by replacing $F$ with $|V(\Bcal)|$ edges of the form $F\cup \{x_i\}$, 
$i=1, \ldots,  |V(\Bcal)|$,  
where $x_i$ is a new vertex for every $i$. Then we add a set of new edges such that on the new vertices they form a hypergraph isomorphic to
$\Bcal$.
\end{defi}

\begin{figure}[!ht]
    \centering
\definecolor{uuuuuu}{rgb}{0.26666666666666666,0.26666666666666666,0.26666666666666666}
\begin{tikzpicture}[line cap=round,line join=round,>=triangle 45,x=0.6cm,y=0.6cm]
\clip(17.767665994357333,-0.7967033860111927) rectangle (40.17626829828812,5.828448599498855);
\draw (22.61745045433849,1.628188843979413) node[anchor=north west] {$F$};
\draw (19.7,5.5) node[anchor=north west] {$\mathcal{A}$};
\draw (24.7,5.5) node[anchor=north west] {$\mathcal{B}$};
\draw (36.4,5.1) node[text width=6cm] {$\mathcal{A}$ extended by $\mathcal{B}$ through $F$}; 

\draw [shift={(25.,4.)},line width=1.pt]  plot[domain=0.:3.141592653589793,variable=\t]({1.*0.2*cos(\t r)+0.*0.2*sin(\t r)},{0.*0.2*cos(\t r)+1.*0.2*sin(\t r)});
\draw [shift={(25.,0.)},line width=1.pt]  plot[domain=3.141592653589793:6.283185307179586,variable=\t]({1.*0.2*cos(\t r)+0.*0.2*sin(\t r)},{0.*0.2*cos(\t r)+1.*0.2*sin(\t r)});
\draw [line width=1.pt] (24.8,0.)-- (24.8,4.);
\draw [line width=1.pt] (25.2,4.)-- (25.2,0.);
\draw [shift={(27.,2.)},line width=1.pt]  plot[domain=-0.7853981633974483:2.356194490192345,variable=\t]({1.*0.3999999999998184*cos(\t r)+0.*0.3999999999998184*sin(\t r)},{0.*0.3999999999998184*cos(\t r)+1.*0.3999999999998184*sin(\t r)});
\draw [shift={(25.,0.)},line width=1.pt]  plot[domain=2.356194490192345:5.497787143782138,variable=\t]({1.*0.3999999999998184*cos(\t r)+0.*0.3999999999998184*sin(\t r)},{0.*0.3999999999998184*cos(\t r)+1.*0.3999999999998184*sin(\t r)});
\draw [line width=1.pt] (24.71715728752551,0.28284271247449055)-- (26.71715728752551,2.2828427124744906);
\draw [line width=1.pt] (27.28284271247449,1.7171572875255094)-- (25.28284271247449,-0.28284271247449055);
\draw [shift={(27.,2.)},line width=1.pt]  plot[domain=-2.356194490192345:0.7853981633974483,variable=\t]({1.*0.3000000000002406*cos(\t r)+0.*0.3000000000002406*sin(\t r)},{0.*0.3000000000002406*cos(\t r)+1.*0.3000000000002406*sin(\t r)});
\draw [shift={(25.,4.)},line width=1.pt]  plot[domain=0.7853981633974483:3.9269908169872414,variable=\t]({1.*0.3000000000002406*cos(\t r)+0.*0.3000000000002406*sin(\t r)},{0.*0.3000000000002406*cos(\t r)+1.*0.3000000000002406*sin(\t r)});
\draw [line width=1.pt] (25.212132034356134,4.212132034356134)-- (27.212132034356134,2.2121320343561344);
\draw [line width=1.pt] (26.787867965643866,1.7878679656438656)-- (24.787867965643866,3.7878679656438656);
\draw [line width=1.pt] (20.433012701892217,3.982050807568876)-- (21.433012701892217,2.25);
\draw [line width=1.pt] (18.566987298107925,2.25)-- (19.566987298107684,3.9820508075689354);
\draw [line width=1.pt] (21.,1.5)-- (19.,1.5);
\draw [shift={(20.,3.7320508075688776)},line width=1.pt]  plot[domain=0.5235987755982981:2.6179938779914913,variable=\t]({1.*0.5*cos(\t r)+0.*0.5*sin(\t r)},{0.*0.5*cos(\t r)+1.*0.5*sin(\t r)});
\draw [shift={(19.,2.)},line width=1.pt]  plot[domain=2.6179938779914913:4.71238898038469,variable=\t]({1.*0.499999999999834*cos(\t r)+0.*0.499999999999834*sin(\t r)},{0.*0.499999999999834*cos(\t r)+1.*0.499999999999834*sin(\t r)});
\draw [shift={(21.,2.)},line width=1.pt]  plot[domain=-1.5707963267948966:0.5235987755983013,variable=\t]({1.*0.5*cos(\t r)+0.*0.5*sin(\t r)},{0.*0.5*cos(\t r)+1.*0.5*sin(\t r)});
\draw [line width=1.pt] (30.,2.4)-- (32.,2.3999999999998183);
\draw [line width=1.pt] (32.,1.6000000000001817)-- (30.,1.6);
\draw [line width=1.pt] (34.,1.5999999999998975)-- (32.,1.6000000000001817);
\draw [line width=1.pt] (36.28284271247449,3.7171572875255094)-- (34.28284271247449,1.7171572875255094);
\draw [shift={(36.,4.)},line width=1.pt]  plot[domain=-0.7853981633974483:2.356194490192345,variable=\t]({1.*0.3999999999998184*cos(\t r)+0.*0.3999999999998184*sin(\t r)},{0.*0.3999999999998184*cos(\t r)+1.*0.3999999999998184*sin(\t r)});
\draw [shift={(30.,2.)},line width=1.pt]  plot[domain=1.5707963267948966:4.71238898038469,variable=\t]({1.*0.4*cos(\t r)+0.*0.4*sin(\t r)},{0.*0.4*cos(\t r)+1.*0.4*sin(\t r)});
\draw [shift={(33.668629150102205,2.8000000000001632)},line width=1.pt]  plot[domain=4.71238898038469:5.497787143782138,variable=\t]({1.*0.4000000000001025*cos(\t r)+0.*0.4000000000001025*sin(\t r)},{0.*0.4000000000001025*cos(\t r)+1.*0.4000000000001025*sin(\t r)});
\draw [line width=1.pt] (32.,2.3999999999998183)-- (33.668629150102205,2.4);
\draw [line width=1.pt] (33.951471862576696,2.5171572875256727)-- (35.71715728752551,4.2828427124744906);
\draw [shift={(34.,2.)},line width=1.pt]  plot[domain=4.71238898038469:5.497787143782138,variable=\t]({1.*0.4000000000001025*cos(\t r)+0.*0.4000000000001025*sin(\t r)},{0.*0.4000000000001025*cos(\t r)+1.*0.4000000000001025*sin(\t r)});
\draw [line width=1.pt] (30.,2.2)-- (32.,2.2);
\draw [line width=1.pt] (32.,1.8)-- (30.,1.8);
\draw [line width=1.pt] (32.,2.2)-- (34.,2.2);
\draw [line width=1.pt] (34.,1.8)-- (32.,1.8);
\draw [line width=1.pt] (34.,2.2)-- (38.,2.2);
\draw [line width=1.pt] (38.,1.8)-- (34.,1.8);
\draw [shift={(38.,2.)},line width=1.pt]  plot[domain=-1.5707963267948966:1.5707963267948966,variable=\t]({1.*0.2*cos(\t r)+0.*0.2*sin(\t r)},{0.*0.2*cos(\t r)+1.*0.2*sin(\t r)});
\draw [shift={(30.,2.)},line width=1.pt]  plot[domain=1.5707963267948966:4.71238898038469,variable=\t]({1.*0.2*cos(\t r)+0.*0.2*sin(\t r)},{0.*0.2*cos(\t r)+1.*0.2*sin(\t r)});
\draw [line width=1.pt] (30.,2.3)-- (32.,2.3000000000002423);
\draw [line width=1.pt] (32.,1.6999999999997577)-- (30.,1.7);
\draw [line width=1.pt] (32.,2.3000000000002423)-- (34.,2.2999999999998635);
\draw [line width=1.pt] (34.21213203435614,2.212132034356138)-- (36.21213203435614,0.21213203435613792);
\draw [shift={(30.,2.)},line width=1.pt]  plot[domain=1.5707963267948966:4.71238898038469,variable=\t]({1.*0.3*cos(\t r)+0.*0.3*sin(\t r)},{0.*0.3*cos(\t r)+1.*0.3*sin(\t r)});
\draw [shift={(36.,0.)},line width=1.pt]  plot[domain=-2.356194490192345:0.7853981633974483,variable=\t]({1.*0.3000000000002456*cos(\t r)+0.*0.3000000000002456*sin(\t r)},{0.*0.3000000000002456*cos(\t r)+1.*0.3000000000002456*sin(\t r)});
\draw [shift={(33.75147186257523,1.400000000000226)},line width=1.pt]  plot[domain=0.7853981633974483:1.5707963267948966,variable=\t]({1.*0.3000000000002456*cos(\t r)+0.*0.3000000000002456*sin(\t r)},{0.*0.3000000000002456*cos(\t r)+1.*0.3000000000002456*sin(\t r)});
\draw [shift={(34.,2.)},line width=1.pt]  plot[domain=0.7853981633974483:1.6643586765480074,variable=\t]({1.*0.3000000000002456*cos(\t r)+0.*0.3000000000002456*sin(\t r)},{0.*0.3000000000002456*cos(\t r)+1.*0.3000000000002456*sin(\t r)});
\draw [line width=1.pt] (32.,1.6999999999997577)-- (33.75147186257523,1.7);
\draw [line width=1.pt] (33.96360389693137,1.6121320343563639)-- (35.78786796564386,-0.21213203435613792);
\draw [shift={(36.,4.)},line width=1.pt]  plot[domain=0.:3.141592653589793,variable=\t]({1.*0.2*cos(\t r)+0.*0.2*sin(\t r)},{0.*0.2*cos(\t r)+1.*0.2*sin(\t r)});
\draw [shift={(36.,0.)},line width=1.pt]  plot[domain=3.141592653589793:6.283185307179586,variable=\t]({1.*0.2*cos(\t r)+0.*0.2*sin(\t r)},{0.*0.2*cos(\t r)+1.*0.2*sin(\t r)});
\draw [line width=1.pt] (35.8,0.)-- (35.8,4.);
\draw [line width=1.pt] (36.2,4.)-- (36.2,0.);
\draw [shift={(38.,2.)},line width=1.pt]  plot[domain=-0.7853981633974483:2.356194490192345,variable=\t]({1.*0.3999999999998184*cos(\t r)+0.*0.3999999999998184*sin(\t r)},{0.*0.3999999999998184*cos(\t r)+1.*0.3999999999998184*sin(\t r)});
\draw [shift={(36.,0.)},line width=1.pt]  plot[domain=2.356194490192345:5.497787143782138,variable=\t]({1.*0.3999999999998184*cos(\t r)+0.*0.3999999999998184*sin(\t r)},{0.*0.3999999999998184*cos(\t r)+1.*0.3999999999998184*sin(\t r)});
\draw [line width=1.pt] (35.71715728752551,0.28284271247449055)-- (37.71715728752551,2.2828427124744906);
\draw [line width=1.pt] (38.28284271247449,1.7171572875255094)-- (36.28284271247449,-0.28284271247449055);
\draw [shift={(38.,2.)},line width=1.pt]  plot[domain=-2.356194490192345:0.7853981633974483,variable=\t]({1.*0.3000000000002456*cos(\t r)+0.*0.3000000000002456*sin(\t r)},{0.*0.3000000000002456*cos(\t r)+1.*0.3000000000002456*sin(\t r)});
\draw [shift={(36.,4.)},line width=1.pt]  plot[domain=0.7853981633974483:3.9269908169872414,variable=\t]({1.*0.3000000000002456*cos(\t r)+0.*0.3000000000002456*sin(\t r)},{0.*0.3000000000002456*cos(\t r)+1.*0.3000000000002456*sin(\t r)});
\draw [line width=1.pt] (36.21213203435614,4.212132034356138)-- (38.21213203435614,2.212132034356138);
\draw [line width=1.pt] (37.78786796564386,1.787867965643862)-- (35.78786796564386,3.787867965643862);
\draw [line width=1.pt] (31.433012701892217,3.9820508075688767)-- (32.43301270189202,2.2499999999998863);
\draw [line width=1.pt] (29.566987298107975,2.24999999999989)-- (30.566987298107975,3.982050807568765);
\draw [line width=1.pt] (32.,1.5)-- (30.,1.5);
\draw [shift={(31.,3.7320508075688776)},line width=1.pt]  plot[domain=0.5235987755982997:2.6179938779914957,variable=\t]({1.*0.5*cos(\t r)+0.*0.5*sin(\t r)},{0.*0.5*cos(\t r)+1.*0.5*sin(\t r)});
\draw [shift={(30.,2.)},line width=1.pt]  plot[domain=2.617993877991491:4.71238898038469,variable=\t]({1.*0.4999999999997768*cos(\t r)+0.*0.4999999999997768*sin(\t r)},{0.*0.4999999999997768*cos(\t r)+1.*0.4999999999997768*sin(\t r)});
\draw [shift={(32.,2.)},line width=1.pt]  plot[domain=-1.5707963267948966:0.5235987755983033,variable=\t]({1.*0.5*cos(\t r)+0.*0.5*sin(\t r)},{0.*0.5*cos(\t r)+1.*0.5*sin(\t r)});
\draw [shift={(19.,2.)},line width=1.pt]  plot[domain=1.5707963267948966:4.71238898038469,variable=\t]({1.*0.4*cos(\t r)+0.*0.4*sin(\t r)},{0.*0.4*cos(\t r)+1.*0.4*sin(\t r)});
\draw [shift={(23.,2.)},line width=1.pt]  plot[domain=-1.5707963267948966:1.5707963267948966,variable=\t]({1.*0.4*cos(\t r)+0.*0.4*sin(\t r)},{0.*0.4*cos(\t r)+1.*0.4*sin(\t r)});
\draw [line width=1.pt] (23.,1.6)-- (19.,1.6);
\draw [line width=1.pt] (19.,2.4)-- (23.,2.4);
\begin{scriptsize}
\draw [fill=black] (19.,2.) circle (1.5pt);
\draw [fill=black] (21.,2.) circle (1.5pt);
\draw [fill=black] (23.,2.) circle (1.5pt);
\draw [fill=black] (25.,4.) circle (1.5pt);
\draw [fill=black] (27.,2.) circle (1.5pt);
\draw [fill=black] (25.,0.) circle (1.5pt);
\draw [fill=uuuuuu] (20.,3.7320508075688776) circle (1.5pt);
\draw [fill=black] (30.,2.) circle (1.5pt);
\draw [fill=black] (32.,2.) circle (1.5pt);
\draw [fill=black] (34.,2.) circle (1.5pt);
\draw [fill=black] (36.,4.) circle (1.5pt);
\draw [fill=black] (38.,2.) circle (1.5pt);
\draw [fill=black] (36.,0.) circle (1.5pt);
\draw [fill=uuuuuu] (31.,3.7320508075688776) circle (1.5pt);
\end{scriptsize}
\end{tikzpicture}

    \caption{An example of the extension operation.}
    \label{fig:extension}
\end{figure}
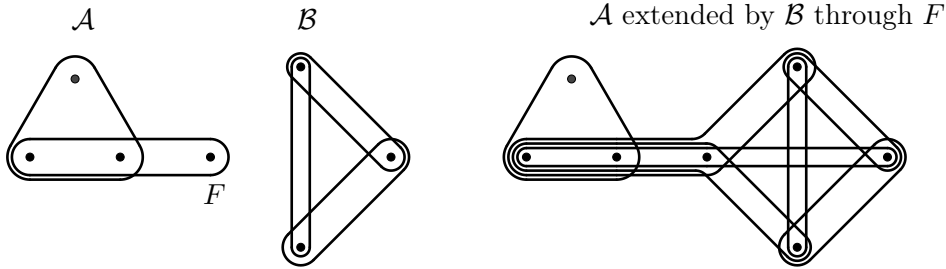

See Figure \ref{fig:extension} for an example. This operation has a number of easily verifiable properties. 

\begin{enumerate}
    \item Suppose that $\Ac$ is not $c$-colorable and $\Bcal$ is not $(c-1)$-colorable. Then no extension of $\Ac$ by $\Bcal$ through any edge $F$ is $c$-colorable.

    \item Suppose also that $\Ac$ is $m$-uniform and $\Bcal$ is $(m+1)$-uniform.
    If we extend $\Ac$ by $\Bcal$ through each edge of $\Ac$ (one by one), we obtain an $(m+1)$-uniform hypergraph which is not $c$-colorable.
\end{enumerate}


For any positive integer $i$, let $\Gc_i$ denote the hypergraph that has $i$ vertices and only a single edge, that contains all the $i$ vertices. Clearly, $\Gc_1$ is not $c$-colorable for any $c$ (because it has an edge with just one vertex) and $\Gc_i$ is not 1-colorable. Hence, we can build non-$c$-colorable hypergraphs starting from these trivial ones and using them in the extensions.    

\begin{obs}
    For any rooted tree $T$ the hypergraph $\Hc_T$ can be built with a sequence of extensions starting from $\Gc_1$, where each extending hypergraph is one of the $\Gc_i$-s.
\end{obs}

If we use more complicated hypergraphs in the extensions, we get hypergraphs with higher chromatic numbers. As an example, we describe a non-$3$-colorable $m$-uniform hypergraph $\Hc^3(m)$ that was used to show that $\chim>3$ for disks in the plane \cite{DP20}. 
We define $\Hc^3(m)$ inductively using $\Hc^2(m)$. 

First, we create a sequence of hypergraphs. Let $\Fcal_1(m)=\Gc_1$ and let $v$ denote the single vertex of it. For $i>1$, let $\Fcal_i(m)$ be the hypergraph that we obtain by extending through 
each edge of $\Fcal_{i-1}(m)$, that contains $v$,  by $\Hc^2(m)$.
(The order of hyperedges of $\Fcal_{i-1}(m)$, that we extend through, does not matter.)
Note that $\Fcal_i(m)$ has only two types of edges. There are the edges that contain $v$; each of these contains exactly $i$ vertices. (These are like the descendant edges.) And there are the edges that were added in a copy of $\Hc^2(m)$; each of these contain exactly $m$ vertices. (These are a bit like the sibling edges.) 
Therefore, $\Hc^3(m):=\Fcal_m(m)$ is $m$-uniform, and by induction, no $\Fcal_i(m)$ is $3$-colorable. 

For example, $\Fcal_1(2)$ is $\Gc_1$ and $\Hc^2(2)$ is a $K_3$. Extending $\Gc_1$ through its single edge by a $K_3$ gives us a $K_4$, so $\Hc^3(2)$ is just $K_4$.  

We could define $\Hc^t(m)$ in a similar manner using $\Hc^{t-1}(m)$ for any $t>2$ to obtain an $m$-uniform hypergraph that is not $t$-colorable.

\subsubsection{Double recursive construction}

Next, we describe an abstract 3-chromatic $m$-uniform hypergraph that first appeared in \cite{P10} and was also used in \cite{PP,DP22,K13}.
Surprisingly, later it was also discovered that the ratio of the hereditary and determinant discrepancy of this hypergraph is logarithmic in the number of vertices, which drew a lot of attention from people studying discrepancy theory; see \cite[Section 5]{matousekdiscrep} for the connection and \cite{linikolov} for the most recent result.

For any positive integers $k$ and $l$, the abstract hypergraph $\Hc(k,l)$ with vertex set $V(k,l)$ and edge set $E(k,l)$ is defined recursively.
The edge set $E(k,l)$ is the disjoint union of two sets, $E(k,l)=E_R(k,l)\cup E_B(k,l)$, where the subscripts $R$ and $B$ stand for red and blue.
All edges belonging to $E_R(k,l)$ are of size $k$, and all edges belonging to $E_B(k,l)$ are of size $l$.
In other words, $\Hc(k,l)$ is the union of a {\em $k$-uniform} and an {\em $l$-uniform} hypergraph.
If $k=l=m$, then we get an $m$-uniform hypergraph.

\begin{defi}\label{def:hypergraph} Let $k$ and $l$ be positive integers.
	\begin{enumerate}
		\item For $k=1$, let $V(1,l)$ be an $l$-element set.\\ Set $E_R(1,\ell) := \{ \{v\} \mid v \in V(1,\ell) \}$ and $E_B(1,l):=\{V(1,l)\}$.
		      
		\item For $l=1$, let $V(k,1)$ be a $k$-element set.\\ Set $E_R(k,1):=\{V(k,1)\}$ and $E_B(k,1):=\{ \{v\} \mid v \in V(k,1)\}$.

		\item For any $k,l>1$, we pick a new vertex $p$, called the {\em root}, and let
		      $$V(k,l):=V(k-1,l)\cup V(k,l-1)\cup\{p\},$$
		      $$E_R(k,l):=\{e\cup \{p\} \,:\, e\in E_R(k-1,l)\} \cup E_R(k,l-1),$$
		      $$E_B(k,l):=E_B(k-1,l)\cup\{e\cup \{p\} \,:\, e\in E_B(k,l-1)\}.$$
	\end{enumerate}
    \begin{figure}[!ht]
        \centering
        \includegraphics[width=0.5\linewidth]{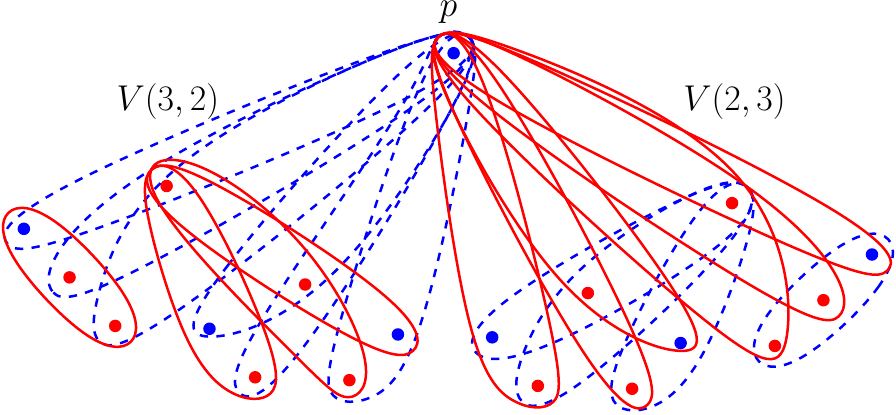}
        \caption{The hypergraph $\mathcal{H}(3,3)$ with a 2-coloring. In this case there is a red edge whose vertices are all red.}
        \label{fig:v33}
    \end{figure}

    The hypergraph $\HH(k,l)$ is not 2-colorable, moreover, in any 2-coloring of the vertices, there is a monochromatic blue edge from $E_B(k,l)$ or a monochromatic red edge from $E_R(k,l)$, see Figure \ref{fig:v33}. 
\end{defi}


    

\subsection{Geometric hypergraphs}\label{sec:geohyp}

Graphs and hypergraphs defined by some geometric objects appear in many different 
contexts. String graphs are the intersection graphs of planar curves. 
They were introduced by Benzer \cite{B59}, a biologist, and they are fundamental in many applications, still, we do not really understand their structure. A breakthrough result of Schaefer and \v{S}tefankovi\'c \cite{SS01}
is that the recognition of string graphs is {\em decidable}.  
In this chapter we investigate geometric hypergraphs and their basic properties.


A hypergraph $\HH=(V,E)$ is called {\em geometric} if its structure is derived from some geometric configuration in some (Euclidean) space.
A family \RR of sets in some Euclidean space is often referred to, especially in computational geometry, as a \emph{range space}. We can associate with \RR a hypergraph whose vertex set $V$ is a subset of the underlying space, and whose edges correspond to the sets in \RR with the natural containment relation. This is called the \emph{primal hypergraph} induced by \RR.
An example of a range space is the family  of all disks 
in the plane, and in its primal hypergraph, $\HH(disks)=(\R^2,disks)$, we have $v\in e$ if the point corresponding to $v$ is contained in the disk corresponding to $e$.

It is often more convenient to consider the hypergraph whose vertices correspond to the sets in \RR, and whose edges correspond to the points of the underlying space, with the \emph{reverse} of the natural containment relation. Here the vertices of the hypergraph correspond to the ranges and the edges correspond to the points of the space, and an edge $e$ consists of all vertices that correspond to ranges which contain the point corresponding to $e$. This is called the \emph{dual hypergraph} induced by \RR.
In the abstract hypergraph-theoretic sense, this is indeed the dual of the primal hypergraph induced by \RR. The dual hypergraph induced by unit disks was the motivating example discussed in the Introduction.
\smallskip

Given two range spaces \RR and \RR', we define the hypergraph $\II(\RR,\RR')$, whose vertices correspond to the sets of \RR and whose hyperedges correspond to the sets of \RR' so that a hyperedge contains a vertex if the corresponding sets intersect.
$\II(\RR,\RR')$ is called the \emph{intersection hypergraph} induced by \RR and \RR'. Note that this is a common generalization of primal and dual hypergraphs by choosing either \RR or \RR' to be a set of points. Using the containment and reverse containment relations, one can define other hypergraphs induced by \RR and \RR'; see \cite[Section 4]{KP15}, where this was done for intervals.
\smallskip

We will mainly consider finite subconfigurations. This means that for a fixed range space, such as for the family of disks in $\R^2$, we will be interested in the hereditary family $\F(disks)$ that consists of all finite subconfigurations of $\HH(disks)$, or its dual, $\F^*(disks)$.
For simplicity, we will call $\F(disks)$ the \emph{family of disks}, and we will also refer to parameters, such as $\chim(\F(disks))$, simply as  $\chim(disks)$, or just write \emph{``for disks, $\chim=4$''} instead of ``$\chim(\F(disks))=4.$'' 
Analogously, instead of parameters of the dual of the family, such as $\chim(\F^*(disks))$, we will simply write \emph{``for disks, $\chim^*$.''}
If a hypergraph \HH belongs to the family of disks, we will also say that \emph{``\HH is realizable by disks''} and, similarly, if $\HH^*$ belongs to the family of disks, then we will say that \emph{``\HH has a dual realization by disks.''}
\smallskip

For a geometric family $\F(\RR)$, a statement of the form $\chim\le k$ is equivalent to the following:
There is an $m=m(\F)$ such that any finite set of points $P$ can be $k$-colored such that for every 
$R\in \RR$ with $|P\cap R|\ge m$, not all the points in $P\cap R$ have the same color.
Similarly, $m_k\le m$ is equivalent to the following:
Any finite set of points $P$ can be $k$-colored such that if for some $R\in \RR$ we have $|P\cap R|\ge m$, then all $k$ colors occur among the points in $P\cap R$.
The statement $m_k^*\le m$ about the dual range space is equivalent to the following:
Any finite collection of sets $R_1,\ldots,R_n\in \RR$ can be $k$-colored such that 
for every point $p$ with $|\{R_i\mid p\in R_i\}|\ge m$, all 
$k$ colors occur among the sets $\{R_i\mid p\in R_i\}$.
The latter can also be rephrased as any finite $m$-fold covering of any subset of the underlying space is decomposable into $k$ disjoint coverings, which was discussed in the Introduction for disks.

\subsubsection{Examples of basic geometric hypergraphs and their (poly)chromatic numbers}\label{sec:examples}

The simplest example to consider is the family of intervals in \R.
A result, sometimes attributed to Tibor Gallai, is that for intervals, we have $m_k=k$ and $m_k^*=k$.
We defined these parameters only for finite hypergraphs, in which case both of these follow by a simple induction.
It is not hard to generalize this result to the family of \emph{$t$-intervals}, that is, sets formed by the union of $t\ge 2$ intervals.
In this case, it is easy to see that $m_k=t(k-1)+1$, which means that these families satisfy the inequality (\ref{m_kkeplet}) with equality for all 
$k, t\ge 2$.

The dual case is entirely different; it is not hard to see that we can realize the $m$-ary tree hypergraph $\Hc^2(m)$ by $2$-intervals, and hence $m_2^*=\infty$ for any $t\ge 2$.
In fact, we can even realize $\Hc^t(m)$, which shows that $\chim^*\ge t+1$, and from Lemma \ref{lem:combine} it follows that this is sharp, so for the union of $t$ intervals $\chim^*=t+1$.
We sketch this upper bound for $t=2$ as a warm-up.

Suppose that we have a finite collection of $2$-intervals, $\{I_i^1\cup I_i^2\}$, 
$1\le i\le N$.
Partition the points of \R covered by at least $5$ members of our family of $2$-intervals into $P_1$ and $P_2$ such that every point $p\in P_1$ is covered by some $I_{i_1}^1$, $I_{i_2}^1$ and
$I_{i_3}^1$, while every point $p\in P_2$ is covered by some $I_{j_1}^2$, $I_{j_2}^2$
$I_{j_3}^2$.
Using Gallai's result, there is a polychromatic $3$-coloring of the family $\{I_i^1\}$ for $P_1$, and another polychromatic $3$-coloring of the family $\{I_i^2\}$ for $P_2$.
Finally, we can combine these two colorings by Lemma \ref{lem:combine} to obtain a proper $3$-coloring, implying $\chim^*\le 3$.
\smallskip

Another classic example is the family of lines in $\R^2$.
It was observed in \cite{PTT09} 
and independently by Pesant \cite{P07}
that a generic projection from a sufficiently high dimensional grid, $\{1,\ldots,m\}^d$, to the plane shows that $\chim=\infty$, using the Hales-Jewett theorem and, by point-line duality, also $\chim^*=\infty$. 

This idea also works for any family that contains arbitrarily thin objects in any direction, such as the family of ellipses or the family of all rectangles. If we restrict ourselves to rectangles that are axis-parallel, this method does not work, but we still have $\chim=\chim^*=\infty$ \cite{CPST09,PacT10}. See Section \ref{sec:axis} for more results on shapes with axis-parallel boundaries. 
\smallskip


Fix a convex set $C$ in the plane and consider  all translates or homothets of $C$. For example, if $C$ is a unit disk, this gives us the family of unit disks and the family of disks in the plane, respectively. In this case, the value of $\chim$ depends on $C$. We will discuss these results in detail in Sections \ref{sec:translates} and \ref{sec:homothet}, where $C$ is a polygon, and in Section \ref{sec:disks}, where $C$ is a disk or a general convex set.

For the family of half-planes in the plane, Smorodinsky and Yuditsky \cite{SY} showed that we always have a polychromatic coloring with $k$ colors with $m=2k+1$, and that this is best possible, so $m_k=2k+1$. See Section \ref{sec:abasummary} for generalizations of this result.

\subsubsection{Relationships between geometric hypergraphs}\label{sec:relations}

In this survey, we want to avoid issues with infinite hypergraphs, that is why we study $\chim$ and $m_k$ for finite subconfigurations for geometric range spaces. For some issues concerning infinite hypergraphs, see \cite[Section 3]{P10}, while for positive results, see \cite{KT14}.
Because of this, we typically do not need to specify whether the underlying sets are open or closed, as it usually follows by a perturbation argument that the finite hypergraphs of the respective families are the same.

For example, finite hypergraphs realizable by open disks are the same as finite  hypergraphs realizable by closed disks, because in any realization after an appropriate perturbation the incidences remain the same, but no point will fall on the boundary of the disks.
In this section, we describe some further equivalences and containments between geometric hypergraphs.
An extensive list and a description of these relationships can be found in
\cite{cogezoo}.

Another important equivalence \cite{pach1986,surveycd} is that if a family is the collection of some (or all) translates of one set in $\R^d$, then the primal and dual hypergraphs induced by the range spaces are isomorphic.
Indeed, consider a family $\C=\{C_i \mid i\in I\}$ of translates of a set $C\subset\mathbb{R}^d$ and a set of points $P\subseteq \mathbb{R}^d$. Suppose, without loss of generality, that $C$ contains the origin $o$. For every $i\in I$, let $c_i$ denote the point of $C_i$ that corresponds to $o\in C$. In other words, we have $\C=\{C+c_i \mid i\in I\}$. Assign to each $p\in P$ a translate of $-C$, the reflection of $C$ about the origin, by setting $C^*_p=-C+p$. Observe that $$p\in C_i \;\; \Longleftrightarrow \;\; c_i\in C^*_p$$
which proves that the same hypergraphs are realizable by the primal and dual range spaces.
In the early papers, more focus was put on studying dual 
hypergraphs, but mainly range spaces of translates were studied, when the two problems are equivalent.

A set is \emph{cover-decomposable} \cite{pach1986} if any sufficiently thick covering of the \emph{whole} plane by the translates of the given set can be decomposed into two disjoint coverings.
Suppose that $S$ is a bounded open set and let ${\cal S}$ be the family of translates of $S$.
Then, by the above argument, $\chim({\cal S})=\chi^*_{\rm big}({\cal S})$, 
and, by a standard 
compactness argument,  $\chim({\cal S})=\chi^*_{\rm big}({\cal S})=2$ implies that $S$ is cover-decomposable.
For more about the connection and results, see \cite{surveycd}.

Next, we show that in some so-called \emph{dynamic} versions of hypergraph coloring problems, families are equivalent to other (normal) hypergraph coloring problems.
In a {\em dynamic hypergraph} the vertex set $V$ is 
ordered as $v_1,\ldots,v_n$, and each prefix $v_1,\ldots,v_i$ induces a certain hypergraph from a given family ${\cal F(\RR)}$ where $\RR$ is a given geometric family (range space). Our dynamic 
hypergraph is the union of these hypergraphs. 
That is, $E=\{\{v_1, v_2, \ldots, v_i\}\cap R\ \colon \ R\in{\cal R}, 1\le i\le n\}$.
A good way to visualize this problem is that the points ``appear'' in the given order.
This dynamic geometric hypergraph is realizable as a standard geometric hypergraph in one dimension higher. 
A simple example is the hypergraph induced by an ordered set of points, $p_1,\ldots,p_n\in \R$, where the range space is formed by all intervals.
The edges of this hypergraph are the sets of the form $I\cap \{p_1,\ldots,p_i\}$ where $I$ is an interval and $1\le i\le n$.
This dynamic interval hypergraph family is, in fact, the same as the family of hypergraphs realizable by so-called bottomless rectangles.
A subset of $\R^2$ is called a bottomless rectangle if it is of the form $\{(x,y): l \leq x \leq r, y \leq t \}$. 
We replace each point $p_i$ by $(p_i, i)$ and each interval $I$ by the bottomless rectangles 
$I\times (-\infty, y)$, for every $0\le y$.
This equivalence was used in \cite{A+13} to prove Theorem \ref{thm:bottomless}, 
and later a similar relationship was used between dynamic quadrants and octants \cite{KP11} to prove Theorem \ref{thm:octant}; 
several further results about dynamic versions can be found in \cite{KP15}.

In the literature, coloring the dynamic hypergraph is usually referred to as quasi-online coloring the (non-dynamic) hypergraph. 
There are also sporadic results about online coloring geometric hypergraphs, for example, online proper colorings of intervals and of quadrants \cite{KLP12}, in which cases the required number of colors depends on the number of vertices, in contrast to the offline and quasi-online versions of these two problems, where a constant number of colors are sufficient.

An easy containment relation among hypergraph families is that hypergraphs realizable by half-spaces (in any dimension) is a subfamily of the hypergraphs realizable by unit balls, which is a subfamily of the hypergraphs realizable by balls.
This latter containment is a special case of the fact that hypergraphs realizable by the translates of a set are always a subfamily of the hypergraphs realizable by the homothetic copies of the same set.
For more results about homothets, see Section~\ref{sec:homothet}.

 \subsubsection{Generalized Delaunay graphs}\label{sec:deluanay}

 
The \emph{generalized Delaunay graph} of a point set $P$ with respect to some bounded convex body $C$ has vertices corresponding to the points of $P$, and two vertices are connected by an edge, if the corresponding two points can be covered by a homothet of $C$ not containing any other point of $P$. The standard Delaunay graph refers to the case when $C$ is a disk.
 It is well-known (see e.g.,\ \cite{Klein}) that the generalized Delaunay graph is connected and planar for any $P$ and $C$ in the plane, 
 moreover, we get a plane drawing if vertices are at the corresponding points of $P$ and edges 
 are drawn as segments.

One can also define the generalized Delaunay graph of an arbitrary hypergraph in an abstract way. Given a hypergraph $\HH$, the generalized Delaunay graph of $\HH$ is the graph on the same vertex set formed by the hyperedges of size two. Note that, given some convex body $C$ and a point set $P$, if we take the primal hypergraph on $P$ induced by the homothets of $C$ and then we take the generalized Delaunay graph of this hypergraph, then indeed we get back the previously defined generalized Delaunay graph.

Assume that we have a hypergraph with the property that every hyperedge contains a hyperedge of size $2$ as a subset. In this case, a proper coloring of its Delaunay graph gives a proper coloring of the hypergraph itself. As primal geometric hypergraphs often have this property, this simple observation will lead to proper $4$-colorability whenever the Delaunay graph is planar, like for disks; more examples can be found in Section \ref{sec:inthyp}.

\section{Translates of polygons}\label{sec:translates}

The study of multiple packings and coverings has a long history \cite{FFK23}.
The main goal  was to find the density of the 
densest packing ($k$-fold packing)
and the thinnest covering ($k$-fold covering) with translates of some convex sets, in particular, polygons and discs. These problems are much easier for lattice arrangements. 

Lagrange proved already in 1773 that the density of the densest lattice packing of the unit disk is the hexagonal packing and it has density $\pi/2\sqrt{3}$.
For general packings, L. Fejes T\'oth proved it in 1942 \cite{F42}.

Early results on cover decomposability focused on translates of polygons. (If we consider \emph{congruent copies} of polygons instead of translates, then the proof of Theorem \ref{thm:classification} implies that $\chim>2$, which in a way shows why they were less interesting.) Let $C$ be a convex polygon in the plane and consider the range space $\mathcal{R}$ induced by the translates of $C$. As we have shown in Section \ref{sec:relations}, for the translates of a fixed set the primal and dual hypergraphs are isomorphic, hence the following results hold  for both the primal and the dual setting.

In a series of papers, it was proved that all open convex polygons are cover-decomposable, that is, $\chim^*(\mathcal{R})=2$. 
Here we mention a few of these papers; for a more detailed overview and for the methods used, see the survey \cite{surveycd}.

\begin{thm}[Pach \cite{pach1986}] For the translates of any centrally symmetric convex polygon, we have $\chim=2$.
\end{thm}

\begin{thm}[Tardos-T\'oth \cite{TT07}]\label{thm:triangle} For the translates of any triangle, we have $\chim=2$.
\end{thm}

\begin{thm}[P\'alv\"olgyi-T\'oth \cite{PT10}]\label{thm:convexpoly} For the translates of any convex polygon, we have $\chim=2$.
\end{thm}

Note that the above statement cannot be extended to all polygons.

\begin{thm}[Pach-Tardos-T\'oth \cite{PTT09}]\label{thm:quadrilateral} For the translates of any concave quadrilateral, we have $\chim>2$.
\end{thm}

Istv\'an Kov\'acs (personal communication) showed that $\chim\le 4$ for translates of any fixed concave quadrilateral $Q$. 
In fact, it can be shown that $\chim=3$ with a somewhat similar method that was used in Section \ref{sec:examples} to show that for the union of $t$ intervals we have $\chi_m^*=t+1$:
Consider $Q$ as a union of two triangles, $T_1$ and $T_2$, find a polychromatic $3$-coloring for $T_1$ and for $T_2$ (these exists by \cite{TT07} or by Theorem~\ref{thm:GV}), then apply Lemma \ref{lem:combine} to get a proper 3-coloring for translates of $Q$ that contain at least $2m-1$ points, thus $\chim=3$.
This argument also shows that $\chim\le n-1$ for the translates of any $n$-gon; for this we need that for the translates of any triangle we can find a polychromatic $k$-coloring for $k=n-1$.
The existence of such a polychromatic $k$-coloring was shown in \cite{TT07} with a constant  
exponential in $k$, and later it was improved to polynomial in \cite{colorful} (see Theorem \ref{thm:homotri}), though now the growth rate is irrelevant for us.


A larger class of concave polygons for whose translates $\chim>2$ was given by Pálvölgyi \cite{P10}, where it turned out that the pairs of angles that can be found in the polygon play an important role.
In fact, based on these results, the translates of polygons have been completely classified with respect to $\chim=2$.

\begin{thm}[P\'alv\"olgyi-T\'oth \cite{PT10}]\label{thm:classification} For the translates of a polygon $\chim=2$ if and only if any pair of its convex angles
are such that either one angle contains the other one, or the same pair could occur in a convex polygon.
\end{thm}

For families with $\chim=2$, the growth rate of the function $m_k$ has also 
been studied extensively.
The only case when $m_2$ was also studied explicitly is the translates of triangles, for which $5\le m_2\le 9$ as a corollary of Theorem \ref{thm:octant} and a lower bound construction \cite{KP15}.
In \cite{pach1986}, it was shown that for any centrally symmetric convex polygon $P$, the parameter $m_k$ exists and is bounded from above by an exponentially fast growing function of $k$.
In \cite{TT07}, a similar result was established for triangles, and in \cite{PT10} for convex polygons.
However, all of these results were improved to the optimal linear bound in a series of papers.

\begin{thm}[Pach-T\'oth \cite{pachtoth2009}] For the translates of any centrally symmetric convex polygon, we have $m_k=O(k^2)$.
\end{thm}

\begin{thm}[Aloupis et al.~\cite{A08}] For the translates of any centrally symmetric convex polygon, we have $m_k=O(k)$.
\end{thm}

\begin{thm}[Gibson-Varadarajan \cite{GV10}]\label{thm:GV} For the translates of any convex polygon, we have $m_k=O(k)$.
\end{thm}



The situation, however, is completely different in higher dimensions.

\begin{thm}[P\'alv\"olgyi \cite{P10}]\label{thm:polyhedron} For the translates of any polytope, we have $\chim>2$.
\end{thm}

The proof is based on the observation that for any polytope $P$, either there is a plane that
intersects $P$ in a concave polygon for which $\chim>2$, because of Theorem \ref{thm:classification}, or there are two parallel
planes that intersect $P$ in two polygons such that taking one convex angle from each, the same pair could not occur in a convex polygon. In both
cases, we can take a plane in space and a family of translates of $P$ that realize the planar construction from $\cite{P10}$ in this plane with the translates of $P$.

\section{Homothets of polygons}\label{sec:homothet}

A homothetic copy, or homothet, of a set is a scaled and translated copy of it (rotations are {\em not} allowed).
We also require the scaling factor to be positive---some other papers would call this a \emph{positive} homothet, but we simply call it a homothet as we never consider negative scalings.
Colorings of geometric range spaces induced by homothetic copies of polygons have been studied much less than translates of polygons.
In this section, every polygon is supposed to be closed.
It follows from the planarity of the generalized Delaunay graph (see Section \ref{sec:deluanay})
that for homothets of a convex polygon we have 
$\chim\le 4$.
The first result about specific homothets was the following.

\begin{thm}[Cardinal et al.\ \cite{colorful}]\label{thm:homotri}
	For the homothets of any triangle we have $m_k\le 144 k^8$.
\end{thm}

Their method was extended by Keszegh and Pálvölgyi \cite{KPself}, who introduced  the following definition.

\begin{defi}
	A collection of closed planar  sets $\cal S$ is {\em self-coverable} if there exists a  
 {\em self-coverability function} $f:\mathbb{N}\rightarrow \mathbb{N}$ such that for any $S \in \cal S$, any $k$  and for any finite point set $P\subset S, |P|=k$ there exists a subcollection ${\cal S}'\subset {\cal S}$, $|{\cal S'}|\le f(k)$  such that $\cup {\cal S'}=S$ but no point of $P$ is in the interior of an $S'\in {\cal S}'$.
\end{defi}

Points outside or on the boundary of $S$ are irrelevant, thus we can assume without loss of generality
that all points of $P$ are in the interior of $S$. It follows directly from the definition that if 
$\cal S$ has a self-coverability function $f$, then it also has a monotone increasing self-coverability function as well.

It is easy to see that (closed)
axis-parallel rectangles are self-coverable with $f(k)=k+1$ and that all disks in the plane (or, in fact, the homothets of any set that is a concave polygon or a set with a smooth boundary) are not self-coverable as already $f(1)$ does not exist, but the homothets of a convex polygon form a self-coverable family, as we will soon see.

Self-coverability can be applied to obtain polychromatic colorings using the following result.

\begin{thm}[Keszegh-P\'alv\"olgyi \cite{KPself}]\label{connection}\label{thm:selfcov}
	Suppose that $\cal S$ is self-coverable with a monotone self-coverability function $f$ for which $f(k)>k$, and that $m_2\le m$ for $\cal S$, i.e., any finite set of points can be colored with two colors such that if $S\in \cal S$ contains at least $m$ points, then $S$ contains both colors.
	Then $m_k\le m\cdot (f(m-1))^{\lceil \log k\rceil-1}$, i.e., any finite set of points can be colored with $k$ colors such that if $S\in \cal S$  contains at least $m\cdot (f(m-1))^{\lceil \log k\rceil-1}\le k^d$ points, then $S$ contains all $k$ colors. (Here $d$ is a constant that depends only on $\cal S$, more specifically, on $m_2$ and $f$.)
\end{thm}

It is not clear whether the technical condition $f(k)>k$ is really necessary here. 
Of course, whenever we have a self-coverability function $f$, then $\max(f(k),k+1)$ is also a self-coverability function by definition, and we can just invoke Theorem \ref{thm:selfcov} for this function if needed (but this way we get a worse bound on $m_k$).


\begin{thm}[Keszegh-P\'alv\"olgyi \cite{KPself}]\label{thm:polygon}\label{thm:selfcover}
The family of all homothets of a given convex polygon $C$ is self-coverable with $f(k)\le ck$, where the constant $c$ depends only on $C$.

Therefore, if $m_2<\infty$ for the homothets of any convex polygon, then $m_k\le k^d$ for some $d$.
\end{thm}


In other words, given a closed convex polygon $C$ and a collection of $k$ points in its interior, we can take $ck$ homothets of $C$ whose union is $C$ such that none of the homothets contains any of the given points in its interior. 
Then, we can use Theorem \ref{thm:selfcov} to obtain a polynomial bound on $m_k$.


For the homothets of triangles and of squares, even the exact value of the optimal $f$ was determined in \cite{KPself}; for triangles, $f(k)=2k+1$ is sharp, and for squares $f(k)=2k+2$ is sharp.
But in general, the constant $c$ in Theorem \ref{thm:polygon} cannot depend only on the number of vertices of the convex polygon $P$ as even for a quadrilateral it might need to be arbitrarily large; for every $c$ there exists a convex quadrilateral $Q$ such that if the family of all homothets of $Q$ is self-coverable with $f$, then $f(k)\ge ck$ \cite{KPself}.


As a corollary of Theorem \ref{thm:9}, the best upper bound for $m_2$ for  
the homothets of a triangle is $9$, and from Theorem \ref{thm:selfcov} we have 
$m_k\le m_2\cdot k^{\log_2 (2m_2-1)}$. Combining these results 
we get $m_k\le 9\cdot  k^{4.09}$ for the homothets of a triangle.
However, apart from the triangle, the square (and its affine images, i.e., parallelograms) is the only convex polygon for whose homothets $m_2<\infty$ has been proved and then using Theorem \ref{connection} we get the following: 

\begin{thm}[Ackerman-Keszegh-Vizer \cite{AKV15}]
For the homothets of any parallelogram we have $m_2\le 215$ and $m_k=O(k^{8.75})$.
\end{thm}

The situation, surprisingly, is completely different for the dual range spaces of homothets.
For the homothets of a triangle we have $m_k^*= O(k^{5.09})$ from Corollary \ref{cor:octant} because if a hypergraph is (dual) realizable by the homothets of a triangle, then it is also realizable by octants. Indeed, the family of intersections of an appropriate plane in the 3 dimensional space with the octants is exactly the family of the homothets of the given triangle in this plane, and so the dual triangle problem follows from the dual octant problem, which is in turn equivalent to the primal octant problem (using that octants are a family of translates of a given region, see Section \ref{sec:relations}).
However, for the homothets of other polygons Kov\'acs showed the following, building on the construction from \cite{P10}.

\begin{thm}[Kov\'acs \cite{K13}]
	For the homothets of any non-triangle polygon, we have $\chim^*>2$, that is, $m_k^*=\infty$.
\end{thm}	

Consequently, for homothets of squares  $\chim=2$ but $\chim^*>2$.
Similarly to translates, the bound $\chim\le n-1$ for the homothets of any $n$-gon follows from Theorem \ref{thm:homotri} and Lemma \ref{lem:combine}, as we have sketched after Theorem \ref{thm:quadrilateral}.
The same bound can also be proved for $\chim^*$, which we sketch here for the homothets of a quadrilateral $Q$; note that this is practically the same argument that we have already used in Section \ref{sec:examples}.

Suppose that we have a finite collection $\QQ$ of homothets of $Q$.
Write $Q$ as the union of two triangles $T_1$ and $T_2$.
Partition the points of the plane covered by at least $2m-1$ members of $\QQ$ ($m$ will be chosen later) into $P_1$ and $P_2$ such that every point $p_i\in P_i$ is covered by at least $m$ homothets of $T_i$, formed by the respective triangles of the quadrangles from $\QQ$.
From Corollary \ref{cor:octant}, if $m$ is large enough, we get a polychromatic 3-coloring of the members of $\QQ$ for the $m$-heavy edges of the dual hypergraph induced by the points of $P_1$, and another polychromatic 3-coloring of the members of $\QQ$ for the $m$-heavy edges of the dual hypergraph induced by the points of $P_2$.
Finally, Lemma \ref{lem:combine} implies $\chim^*\le 3$.

The bound $\chim^*\le n-1$ can be improved for the homothets of {\em convex} polygons and for even more general hypergraphs to $\chim^*\le 4$; see Section \ref{sec:inthyp}.
The bound $\chim\le n-1$ was also improved for {\em convex} polygons as follows.

\begin{thm}[Keszegh-P\'alv\"olgyi \cite{KPpropcol}]\label{thm:main3}
For the homothets of a convex polygon, we have $\chim\le 3$.
\end{thm}

It is not clear whether this result is sharp or not.

\begin{problem}\label{prob:homo}
Is it true that for the homothets 
of any convex polygon, we have $\chim=2$? 
\end{problem}

As we have seen, this statement holds for triangles and squares, and it is also known in the special case when all considered homothets have a common point \cite{DP20}.
It was also shown in \cite{DP20} that 
Theorem \ref{thm:main3} cannot be extended to all planar convex sets.

\section{Axis-parallel boundaries}\label{sec:axis}
Other natural families are those defined by axis-parallel boundaries.
A {\em positive quadrant} is a set of the form $[x_0,\infty)\times [y_0,\infty)$.
The following statement was implicitly proved in \cite{PT10} using a notion called path decomposition.\footnote{In fact, since Berge \cite{Ber72} proved that $m_2=2$ implies $m_k=k$ for all $k$, it would have been enough to prove this statement for $k=2$ but this connection was not discovered at the time. Also, the proof in \cite{PT10} is essentially the same for all $k\ge 2$.}

\begin{claim}[Pálvölgyi-Tóth \cite{PT10}]
For the family of positive quadrants, we have $m_k=m_k^*=k$.
\end{claim}


If instead of positive quadrants, we allow all four quadrants with axis-parallel boundaries, 
then $m_k=O(k)$ is a simple corollary of Theorem \ref{thm:GV} for squares, while $m_k^*\le 4k-3$ follows from the previous claim by applying it to the four families separately.

Another simple case is the family of {\em axis-parallel strips}, 
i.e., the sets of the form $[x_1,x_2]\times \R$ and $\R\times [y_1,y_2]$.
If only horizontal or only vertical strips are considered, then the obtained family is isomorphic to the family of intervals in \R.
If both are allowed, we have the following results.

\begin{thm}[Aloupis et al.~\cite{A09}]\label{thm:strips} For axis-parallel strips $1.5k-1\le m_k\le 2k-1$ and $k\le m_k^*\le 2k-1$.
\end{thm}

In the same paper some generalizations to higher dimensions were also proved.
A {\em slab} in  $\R^d$ is a set of the form $\{ (x_1, x_2, \ldots, x_d)\ \colon \ l\le x_i\le r\}$ for some $1\le i\le d$ and $l\le r$, i.e. the points between two parallel axis-aligned hyperplanes.
It was shown \cite{A09} that 
for $d$-dimensional slabs 
$2\lceil(2d-1)k/2d \rceil\le m_k\le k(4 \ln k + \ln d)$ and $\lfloor k/2\rfloor d+1\le m_k^*\le d(k-1)+1$.

Recall that a (closed) {\em bottomless rectangle} 
is a set $\{(x,y): l \leq x \leq r, y \leq t \}$ for some $l\le r$ and $t$. This family turned out to be of 
particular interest. 
Keszegh \cite{K11} showed that $m_2=4$ and $m_2^*=3$ (implicitly) using the notion of dynamic hypergraphs.
Later, the following general upper bound was given for $m_k$ using dynamic hypergraphs.

\begin{thm}[Asinowski et al.~\cite{A+13}]\label{thm:bottomless} For bottomless rectangles $1.67k-2.5\le m_k\le 3k-2$.
\end{thm}

Therefore, bottomless rectangles are known to satisfy inequality (\ref{m_kkeplet}) 
that is, $m_k(\F)\le (k-1)(m_2(\F)-1)+1$,
but we do not know whether the inequality is tight or not.
Much less is known about their dual range space.
The best bound, $m_k^*=O(k^{5.09})$, follows from the respective result for octants. Indeed, similarly to the dual result about homothets of triangles, choosing an appropriate plane in the 3 dimensional space, the intersections of octants with this hyperplane are the bottomless rectangles of this plane and then we can apply Corollary \ref{cor:octant}. For better bounds in some special cases and for the reason why known methods fail, see \cite{blessnew}.
It was recently proved that for the union of bottomless rectangles and horizontal strips $\chim>2$ \cite{chekan}, while from Theorems \ref{thm:strips}, \ref{thm:bottomless} and Lemma \ref{lem:combine} we get that $\chim\le 3$, thus $\chim=3$.

Observe that if a hypergraph is realizable by bottomless rectangles, then it is also realizable by the homothetic copies of a fixed triangle, as the sides of the bottomless rectangles can be slightly rotated, all by the same angle so that they all become homothetic copies of a fixed triangle $T$. 
Then we can apply an affine transformation to turn $T$ to any triangle.  
This shows that it is harder to prove upper bounds for homothetic copies of triangles than for bottomless rectangles.

For the family of axis-parallel rectangles in the plane, we only have negative results.

\begin{thm}[Chen-Pach-Szegedy-Tardos \cite{CPST09}]\label{thm:rect}
	For axis-parallel rectangles $\chim=\infty$.
\end{thm}

\begin{thm}[Pach-Tardos \cite{PacT10}]\label{thm:rectdual}
	For axis-parallel rectangles $\chim^*=\infty$.
\end{thm}

\begin{remark}\label{remark:ap}
    A lot of research was done on determining how fast the required number of colors grows with $n$.
    In case of
    the dual problem, $\Theta(\log n)$ colors are needed and enough for proper coloring axis-parallel rectangles with respect to points (for any fixed $m$) \cite{Har-PeledS05,PacT10}. 
    In the primal problem, if $m\ge 3$, then again $O(\log n)$ colors are enough \cite{ap13}, yet when $m=2$, i.e., when we color the points such that axis-parallel rectangles containing two points are already properly colored, then it is only known that $\Omega(\log n/\log \log n)$ colors are needed \cite{CPST09} and $O(n^{0.368})$ are sufficient \cite{ajwani, C12}. 
\end{remark}    

\begin{problem}
    What is the minimum number of colors, $f(n)$, with the property that any set of $n$ points in the plane can be $f(n)$-colored such that any axis-parallel rectangle containing at least two points contains points with two different colors.    
\end{problem}

We call the $d$-dimensional generalization of quadrants {\em $d$-quadrants} which is 
a set
$$\{ (x_1, x_2, \ldots, x_d)\ \colon \ l_i\le x_i, \forall i, 1\le i\le d\}$$  for some $l_1, l_2, \ldots, l_d$, and $3$-quadrants are also called orthants.

Cardinal (personal communication in \cite{KP11}) noticed that $4$-quadrants  can simulate the axis-parallel rectangles of an appropriate subplane of $\R^4$, so both Theorems \ref{thm:rect} and \ref{thm:rectdual} imply the following.

\begin{cor}[Cardinal (personal communication)]\label{thm:hex2rec} 
For $4$-quadrants, we have $\chim=\infty$.
\end{cor}
\begin{proof}[Proof (using Theorem \ref{thm:rect})] For any $k$ and $m$, there is a finite planar point set $S$ such that for every $k$-coloring of $S$ there is an axis-parallel rectangle that contains exactly $m$ points of $S$, all of the same color.
	Place this construction on the $2$-dimensional 
    plane $\Pi=\{(x,y,z,w)\mid x+y=0$, $z+w=0\}$ in  $\R^4$.
	A $4$-quadrant  $\{(x,y,z,w)\mid x\ge x_0, y\ge y_0, z\ge z_0, w\ge w_0\}$ intersects $\Pi$ in $\{(x,y,z,w)\mid x_0\le x=-y\le -y_0, z_0\le z=-w\le -w_0\}$, which is a rectangle whose sides are parallel to the lines $\{x+y=0, z=w=0\}$ and $\{x=y=0, z+w=0\}$, respectively.
	Taking these perpendicular lines as axes, any ``axis-parallel'' rectangle of $\Pi$ is realizable by an appropriate $4$-quadrant, and the theorem follows.
\end{proof}

Kolja Knauer (personal communication) observed that all axis-parallel rectangles of a subplane of $\R^3$ can be cut out in a similar way by the homothets of a (regular) tetrahedron.
Indeed, let $\Delta$ be the tetrahedron whose vertices are $(1,0,1),(-1,0,1),(0,-1,-1),$ and $(0,1,-1)$.
Let $\Delta_h$ be a translate of $\Delta$ by $-1<h<1$ parallel to the $z$-axis.
The intersection of $\Delta_h$ with the plane $\Pi=\{(x,y,z)\mid z=0\}$ yields an axis-parallel rectangle $R_h=\Delta_h\cap \Pi$.
The ratio of the sides of $R_h$ depends on $h$, and can take any value, as it tends to $\pm \infty$ as $h\to \pm 1$.
It follows that by taking a homothetic scaling of $\Delta_h$, we can obtain any axis-parallel rectangle.
Just like in the proof of Corollary \ref{thm:hex2rec},
we obtain by Theorem \ref{thm:rect} that for any $c$ and $m$ there is a finite point set $S\subset \R^3$ such that for every $c$-coloring of $S$ there is a homothet of $\Delta$ that contains exactly $m$ points of $S$, all of the same color.
Therefore, $\chim=\infty$ for homothets of simplices.
The same argument, with replacing Theorem \ref{thm:rect} by Theorem \ref{thm:rectdual}, gives $\chim^*=\infty$.

These arguments settle the cases of quadrants and $d$-quadrants for $d\ge 4$. For $3$-quadrants, that is, octants, the best known result is the following, proved using dynamic hypergraphs.

\begin{thm}[Keszegh-P\'alv\"olgyi \cite{KP11,KP15}]\label{thm:octant}\label{thm:9}
	For positive octants $5\le m_2\le 9$.
\end{thm}

While we do not know the exact value of $m_2$, it was shown in \cite{BBCP19} that 6-heavy hypergraphs realizable by octants are always proper 3-colorable.

For $m_k$ for $k\ge 3$, no upper bound at all followed directly from the method for bounding $m_2$ in \cite{KP11}.
Later, only the very weak bound of $m_k\le 12^{2^k}$ was proved \cite{KP12}.
Then the general method of self-coverability (see Section \ref{sec:homothet}) was developed in a series of papers by Cardinal et al.\ \cite{colorful,colorful2} and by Keszegh and Pálvölgyi \cite{KPself} which culminated in the following polynomial bound.

\begin{thm}[Cardinal et al.\ \cite{colorful2}]\label{thm:multioctant}
	For positive octants $m_k\le m_2\cdot k^{\log_2 (2m_2-1)+1}$.
\end{thm}

Note that there is a '+1' in the exponent compared to the bound of Theorem \ref{thm:selfcov}.
This is because octants are not self-coverable, some additional observations were needed; for the details, see \cite{colorful2}.

The combination of Theorems \ref{thm:octant} and \ref{thm:multioctant} implies the following bound.

\begin{cor}\label{cor:octant}
	For positive octants $m_k=O(k^{5.09})$.
\end{cor}

Note that any hypergraph \HH realizable by the homothets of a fixed triangle is also realizable by octants. For regular triangles it follows by embedding the plane realizing \HH with triangles into $\mathbb R^3$ as the $x+y+z=0$ plane (see Figure \ref{fig:oct2homtri}).
As mentioned earlier, the best upper bound for $m_k^*$ for homothets of a triangle, or even for bottomless rectangles, comes from Corollary \ref{cor:octant}.

\begin{figure}[ht]
	\begin{center}
		\scalebox{.8}{\includegraphics{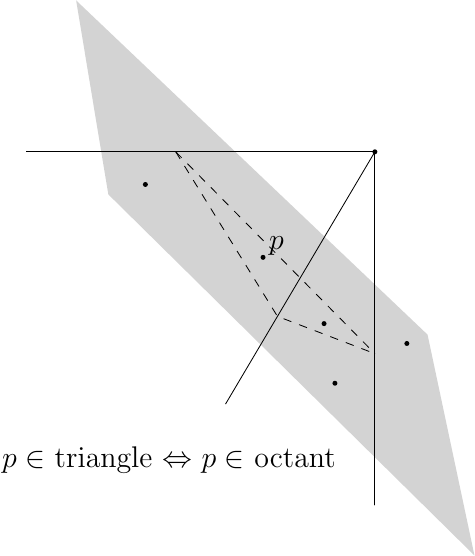}}
	\end{center}\caption{Octants give a richer family than homothetic copies of a triangle, because every homothet of the triangle depicted on the shaded plane can be obtained as the intersection of an octant with the plane.}\label{fig:oct2homtri}
\end{figure}

Finally, we mention that if instead of positive octants, we allow all 8 octants, then it follows from the methods of \cite{P10} that $\chim>2$, while the boundedness of $\chim$ follows from Lemma \ref{lem:combine}, or from simply taking the direct products of the colorings for different octants.

\section{Disks}\label{sec:disks}


As we mentioned in the Introduction, 
one of the starting points of this research in general was Problem \ref{thm:question}, namely, 
of Pach \cite{pach1980}:
Is it true that a sufficiently thick 
covering of the plane by unit disks can be decomposed into two coverings? In our terminology, the question is practically\footnote{They are not entirely equivalent, as we consider finite hypergraphs, while Pach's question is about an infinite covering. But, as expected, the same solution worked for both problems.} equivalent to whether $m_2$ is finite or not for unit disks. 
It was finally answered in the negative, in a general form.

\begin{theorem}[Pach-Pálvölgyi \cite{PP}]\label{thm:unsplittable}
	Let $C$ be any convex body in the plane that has two parallel supporting lines such that $C$ is strictly convex in some neighborhood of the two points of tangencies. Then for any positive integer $m$, there exists a 3-chromatic $m$-uniform hypergraph that is realizable with translates of $C$, therefore, $\chim>2$.
\end{theorem}

However, for many other convex sets it is not known if for its translates we have $\chim=2$ or not, for example, this problem is open for half-disks.\footnote{Note that in \cite{DP22} it was claimed that strict convexity in some neighborhood of only one point is sufficient, but that proof was incorrect.}


Recall that if $\RR$ consists of the translates or homothets of some planar convex body, then by the planarity of generalized Delaunay graphs and the Four Color Theorem we have $\chim\le 4$.
This was recently improved in the following theorem.

\begin{theorem}[Dam\'asdi-Pálvölgyi \cite{DP22}]\label{thm:general_three_col}
For any convex body $C$ in the plane there is a positive integer $m$ such that any finite point set $P$ in the plane can be three-colored in a way that there is no translate of $C$ containing at least $m$ points of $P$, all of the same color, therefore, $\chim\le 3$ for translates of $C$.
\end{theorem}

This was known when $C$ is a polygon (in which case 2 colors suffice according to Theorem \ref{thm:classification}, and 3 colors are known to be enough even for homothets according to Theorem \ref{thm:main3}) and for pseudo-disk families that intersect in a common point \cite{abab} (which generalizes the case when $C$ is unbounded, in which case 2 colors suffice \cite{PP}).

Note that an earlier version of the proof \cite{DP21} only worked when $C$ is a disk, and while the generalization to other convex bodies with a smooth boundary seemed feasible, there was no direct way to extend it to arbitrary convex bodies.
The proof of Theorem \ref{thm:general_three_col} relies on a surprising connection to two other famous results, the solution of the two-dimensional case of the Illumination conjecture (Levi \cite{Levi}), and a recent solution of the Erdős-Sands-Sauer-Woodrow conjecture by Bousquet, Lochet, and Thomassé~\cite{esswproof}. 
More precisely, in \cite{DP22} they proved and used a generalization of the latter result.

About homothetic copies, $\chim\le 4$ turned out to be sharp for most convex bodies, including the disk, disproving earlier conjectures \cite{K11,KPpropcol}, and improving \cite{PTT09}, which established $\chim>2$ for disks.

\begin{theorem}[Dam\'asdi-Pálvölgyi \cite{DP20}]\label{thm:dp20}
	Let $C$ be any convex body in the plane that has two parallel supporting lines such that $C$ is strictly convex in the neighborhoods of the two points of tangencies. For any positive integer $m$, there exists a $4$-chromatic $m$-uniform hypergraph that is realizable with homothets of $C$, therefore, $\chim=4$ for homothets of $C$.
\end{theorem}

What about the dual problem for disks? For unit disks the problem is self-dual, so Theorem \ref{thm:unsplittable} implies that $\chi^*_{\rm big}>2$ for unit disks and thus also for disks. On the other hand, Smorodinsky \cite{smorod-onthe}, using a dualization argument (by mapping points to half-spaces in $\R^3$ and disks to points in $\R^3$) showed that $\chi^*$ (and so $\chi^*_{\rm big}$ as well) is at most $4$ for disks.

\begin{problem}\label{prob:dualdisk}
	For the homothets of a disk is $\chi^*_{\rm big}=3$ or $4$? 
\end{problem}

Balls and half-spaces in higher dimensions have not yet been studied in detail.
We are aware of only two unpublished observations, made by the authors.

Dam\'asdi noted that for half-spaces in $\mathbb R^3$ we have $4\le \chim\le \chi\le 5$.
The lower bound follows from Theorem \ref{thm:dp20} by projecting up the point set to a paraboloid, while the upper bound follows from coloring the points inside the convex hull with one color, and the points of the convex hull with four more, using the planarity of the Delaunay graph on the surface.

Pálvölgyi noted that half-spaces in $\mathbb R^5$ can cut out any conic section (given by $Ax^2+Bxy+Cy^2+Dx+Ey+F\le 0$) from the surface $x_3=x_1^2, x_4=x_1x_2, x_5=x_2^2$, and thus the Hales-Jewett hypergraph is realizable, for example by thin ellipses, as mentioned in Section \ref{sec:examples}, implying $\chim=\infty$. There are other constructions as well that are realizable in a high enough dimension, but nothing else is known about $\chim$ in $\mathbb R^3$ or in $\mathbb R^4$ for half-spaces and (unit) balls. The following question is particularly interesting. 

\begin{problem}
    For balls in $\R^3$ is $\chim<\infty$?
\end{problem}


\section{Intersection hypergraphs}\label{sec:inthyp}

The previous coloring problems about disks and about homothets of other convex sets can be generalized in two natural ways. First, instead of homothets, we can consider families of pseudo-disks.
	A \emph{Jordan region} is a (simply connected) closed bounded region whose boundary is a closed simple Jordan curve. A family of Jordan regions is called a family of \emph{pseudo-disks} if the boundaries of every pair of the regions intersect in at most two points.
Probably the most studied pseudo-disk family is the family of homothets of a (strictly) convex region.\footnote{Note that the family of all homothets of a convex polygon is actually alredy not a pseudo-disk family, yet if the point set and the family is both finite then we can easily perturb them to be.} 
As we have seen, $\chim=\chi=4$ already for disks. That is, $4$ colors are needed for any restriction on the size of the hyperedges, yet enough even if we consider all hyperedges.
In this section we concentrate on $\chi$, that is, we color all hyperedges, not just the large enough hyperedges. 

The second generalization is to consider intersection hypergraphs, which is 
a common generalization of primal point coloring and dual region coloring problems. 
Recall from Section \ref{sec:geohyp} that the intersection hypergraph $\II(\RR,\RR')$ denotes the hypergraph whose vertices correspond to 
the sets of \RR and hyperedges correspond to the sets of \RR' 
and a hyperedge contains a vertex if the corresponding sets intersect. If \RR (resp. \RR') is a family of points, then we get a primal (resp. dual) hypergraph.

There are much fewer results for intersection hypergraphs than for
primal and dual hypergraphs. In \cite{KP15} intersection (and also inclusion and reverse-inclusion) hypergraphs of intervals of the line were considered by Keszegh and Pálvölgyi. 
In \cite{smorod-int} and \cite{fekete-int} intersection hypergraphs (and graphs) of 
(unit) disks, pseudo-disks, squares and axis-parallel rectangles were considered.

The initial research in this direction was on intersection hypergraphs of disks, in order to state common generalizations of primal and dual results. In this section, we show that this is indeed possible by presenting multiple positive results of this type.

\bigskip
Recall that for any primal hypergraph its generalized Delaunay graph is a graph on the same vertex set and the edges are the 
size $2$ hyperedges of the hypergraph. 
We have seen earlier that the Delaunay graph of points with respect to disks is planar, consequently proper $4$-colorable. Therefore, as every disk contains a disk containing exactly two points, the primal hypergraph for disks is also proper $4$-colorable. 
Consider now a finite set of points and a family of pseudo-disks.
In the primal hypergraph vertices correspond to points, the hyperedges correspond to the pseudo-disks. By the sweeping argument of Snoeyink and Hershberger \cite{SH91} we can extend the pseudo-disk family such that every pseudo-disk contains a pseudo-disk that contains exactly two points. 
It is well known that the generalized Delaunay graph of this hypergraph is still planar (see for example \cite{pinchasi}). Therefore, the hypergraph is proper $4$-colorable.

The dual proper $4$-colorability statement for disks was proved by Even et al. \cite{evenlotker} and was generalized to pseudo-disks by Smorodinsky \cite{smorod-onthe}, who showed that given a finite set of points $V$ and a finite family of pseudo-disks $\B$, $\II(\B,V)$ can be properly colored with constant many colors. For the special case of homothets of a convex region, Cardinal and Korman \cite{cardinalkorman} showed that $4$ colors are enough, just like for disks. It was then proved by Keller and Smorodinsky \cite{smorod-int} that disks with respect to disks can also be colored in such a way, that is, if $\D$ is the family of all disks in the plane and $\B$ a finite family of disks, $\I(\B,\D)$ admits a proper coloring with $6$ colors. Finally, this was further generalized by Keszegh to pseudo-disks, also improving the number of colors to the optimal four.

\begin{thm}[Keszegh \cite{psdiskwrtpsdisk}]\label{thm:psdiskwrtpsdisk}
	Given a family $\F$ of pseudo-disks and a finite family $\B$ of pseudo-disks, $\I(\B,\F)$
	admits a proper coloring with $4$ colors.
\end{thm}

As before, the first part of the proof is to show that the generalized Delaunay graph (that is, the size-$2$ hyperedges) of the hypergraph defined by one family of pseudo-disks with respect to  another family of pseudo-disks is planar. Observe, however, that in this case it is not enough in itself to guarantee proper $4$-colorability of the Delaunay graph, as it is possible that a hyperedge of the intersection graph does not contain a hyperedge of size $2$ already in the special case when $\B$ is a set of points.
Note that in Theorem \ref{thm:psdiskwrtpsdisk} $\B$ and $\F$ are not related in any way, a member of $\B$ and $\F$ can have an arbitrarily complex intersection. 
Thus, it implies the following somewhat surprising statement.

\begin{corollary}[Keszegh \cite{psdiskwrtpsdisk}]
	We can properly color with $4$ colors the family of homothets of a convex region $A$ with respect to the family of homothets of another convex region $B$.
\end{corollary}


We mention two generalizations of the above results, one for non-piercing regions and one for regions with linear union complexity. 

A family $\B$ of Jordan regions is called \emph{non-piercing} if $A\setminus B$ 
is connected for every pair of sets $A,B\in \B$. Observe that a family of 
pseudo-disks is always non-piercing. Raman and Ray \cite{Raman2020} 
investigated respective problems about non-piercing families of regions. 


\begin{thm}[Raman-Ray \cite{Raman2020}]
	\label{thm:non-piercing}
	Given two non-piercing families $\F$ and $\B$ of regions, we can properly color with $4$ colors $\F$ with respect to $\B$. 
\end{thm}

We remark that this is an implication of the main result of \cite{Raman2020} that states that the respective intersection hypergraph admits a so-called planar support, whose definition we omit here.



%

Let $\B$ be a family of finitely many Jordan regions in the plane such that the boundaries of its members intersect in finitely many points. 
    The \emph{union complexity} $\U (\B)$ of $\B$ is the number of intersection points of the arrangement of $\B$ that lie on the boundary of $\bigcup\B $.
	We say that a family of regions $\B$ has linear union complexity if there exists a constant $c$ such that for any subfamily $\B'$ of $\B$ the union complexity of $\B'$ is at most $c|\B'|$.

Kedem et al. \protect\cite{Kedem1986} showed that any finite family of pseudo-disks in the plane has linear union complexity.
The result of  Smorodinsky \cite{smorod-onthe}, mentioned above, follows from a more general statement about families with linear union complexity.

\begin{thm}[Smorodinsky \protect\cite{smorod-onthe}]\label{thm:linunionwrtpt}
	Let $P$ be the set of all points of the plane and let $\B$ be a finite family of Jordan regions with linear union complexity. Then $\I(\B,P)$ admits a proper coloring with a constant number of colors.\footnote{When a statement is about a family with linear union complexity, by a constant we mean a constant that depends on $c$ in the definition of linear union complexity.}
\end{thm}

This was generalized to intersection hypergraphs in the following way:

\begin{thm}[Keszegh \cite{psdiskwrtpsdisk}]\label{thm:linunionwrtpsdisc}
	Given a family $\F$ of pseudo-disks and a finite family $\B$ of Jordan regions with linear union complexity, $\I(\B,\F)$ admits a proper coloring with a constant number of colors.
\end{thm}

\section{ABA-free hypergraphs and generalizations}\label{sec:abasummary}

The first important positive result about shapes whose boundary does not consist of fixed direction segments was the following (see also \cite{K11} for the case $k=2$).

\begin{theorem}[Smorodinsky-Yuditsky \cite{SY}]\label{thm:halfplanes}
	For half-planes $m_k=2k-1$ and $2k-1\le m_k^*\le 3k-2$.
\end{theorem}

Later, Fulek \cite{F11} showed that for the dual range space of half-planes the lower bound is sharp for $k=2$, i.e., $m_2^*=3$.
Keszegh and Pálvölgyi \cite{KP14} managed to generalize Theorem \ref{thm:halfplanes} to so-called pseudo-halfplane arrangements as well.
A \emph{pseudoline arrangement} is a collection of simple bi-infinite $x$-monotone curves 
(that is, curves in the form $y=f(x)$) 
such that any two curves intersect at most once.
A \emph{pseudo-halfplane} is the region on one side of a pseudoline in such an arrangement.
Given a planar point set and a pseudo-halfplane arrangement, they induce a (primal) hypergraph whose vertices are the points and whose edges are the points that are contained in a pseudo-halfplane.
We call hypergraphs that can be obtained in such a way {\em pseudo-halfplane hypergraphs}.
These can also be described in a purely combinatorial way as follows.

\begin{definition}[Keszegh and Pálvölgyi \cite{KP14}]\label{def:aba}

Two subsets $A,B$ of an ordered set of elements form an \emph{$ABA$-sequence} if there are $3$ elements, $a_1 < b < a_2$ such that $\{a_1,a_2\} \subset A \setminus B$ and  $b\in B \setminus A$.
A hypergraph with an {\em ordered} vertex set is {\em ABA-free} if it does not contain two hyperedges $A$ and $B$ that form an $ABA$-sequence.	A hypergraph with an {\em unordered} vertex set is {\em ABA-free} if there is an order of its vertices such that the hypergraph with this ordered vertex set is ABA-free.    	
\end{definition}

ABA-free hypergraphs include several geometric families: primal hypergraphs of\\
i) intervals in $\R$,\\
ii) translates of an upwards unbounded convex set in $\R^2$,\\
iii) upwards half-planes in $\R^2$.

More generally, it was shown in \cite{KP14} that a hypergraph $\HH$ on an {\em ordered} vertex set $S$ is a {pseudo-halfplane hypergraph} if and only if there exists an ABA-free hypergraph $\FF$ on $S$ such that $\HH\subset\FF\cup \bar{\FF}$, where $\bar{\FF}$ is the family of the complements of the hyperedges of $\FF$.	
It was also proved in \cite{KP14} that the dual of an ABA-free hypergraph is also ABA-free. This is similar to the standard point-line duality argument, where the set of primal and dual hypergraphs realizable by points and upwards half-planes are the same, where by upwards half-plane we mean a half-plane $H=\{(x,y)\mid \alpha x + y > c\}$ for some $\alpha>0$ and $c$.


Using the abstract concept of ABA-free hypergraphs, the following were proved.

\begin{thm}[Keszegh-Pálvölgyi \cite{KP14}]\label{thm:primalpshpsurvey}
	For pseudo-halfplanes $m_k=2k-1$ and $2k-1\le m_k^*\le 3k-2$.\\ That is, given a pseudo-halfplane hypergraph ${\mathcal H}$ we can color its vertices with $k$ colors such that every hyperedge $A\in {\mathcal H}$ whose size is at least $2k-1$ contains all $k$ colors, and given a pseudo-halfplane hypergraph ${\mathcal H}$ we can color its hyperedges with $k$ colors such that every vertex which is contained in at least $3k-2$ hyperedges is contained in hyperedges of all $k$ colors.
\end{thm}

It was additionally shown by Fulek \cite{F10} that $m^*_2=3$, so for $k=2$ the lower bound is sharp.

Definition \ref{def:aba} can also be naturally extended to a higher number of alternations, for example, by forbidding $4$ elements, $a_1 < b_1 < a_2 < b_2$ such that $\{a_1,a_2\} \subset A \setminus B$ and  $\{b_1,b_2\} \subset B \setminus A$, we obtain ABAB-free hypergraphs.
These would correspond to hypergraphs whose vertices are planar points and whose edges are determined by regions lying above bi-infinite $x$-monotone curves that pairwise intersect at most twice.
It was also shown in \cite{abab} that a hypergraph is ABAB-free if and only if it is realizable by points with respect to a family of stabbed\footnote{A family is \emph{stabbed} if their intersection is nonempty.} pseudo-disks.

ABAB-free $m$-uniform hypergraphs might not be $2$-colorable \cite{KP14}, but they are always $3$-colorable \cite{abab}.
This implies the following.

\begin{thm}[Ackerman-Keszegh-Pálvölgyi \cite{abab}]\label{thm:mainstabbed}
	Let $\F$ be a family of pseudo-disks whose intersection is non-empty and let $S$ be a finite set of points. Then we can $3$-color the points of $S$ such that any pseudo-disk in $\F$ that contains at least two points from $S$ contains two points of different colors.
	Moreover, for every integer $m$ there is a set of points $S$ and a family of pseudo-disks $\F$ with a non-empty intersection, such that for every $2$-coloring of the points there is a pseudo-disk containing at least $m$ points, all of the same color.	
\end{thm} 


Theorem~\ref{thm:mainstabbed} implies $\chim\le 3$ for stabbed disks, that is, given a finite set of points $S$ and family of stabbed disks, then it is possible to color the points of $S$ with three colors such that any disk that contains at least two points from $S$ contains two points with different colors. Damásdi and Pálvölgyi \cite{DP20} later proved by realizing the $m$-ary tree that this is best possible, i.e., $\chim=3$ for stabbed disks. They also showed that for stabbed unit disks $2$ colors are enough. 

\medskip

Duals of ABAB-free hypergraphs are equivalent to dual hypergraphs induced by stabbed pseudo-disks, i.e., vertices correspond to pseudo-disks and each point corresponds to a hyperedge that consists of the pseudo-disks containing the point. Keszegh and Pálvölgyi \cite{dualabab} proved that such hypergraphs are proper $4$-colorable, while for every $m \ge 2$ there exists a dual of an ABAB-free $m$-uniform hypergraph which is not $2$-colorable. It is not known whether $3$ colors suffice or not. 

\begin{problem}
	Is there a dual ABAB-free hypergraph with $\chim= 4$, or do we always have $\chim\le 3$?\\
    What about the special case of stabbed disks?
\end{problem}

Note that the latter question is also a special case of Problem \ref{prob:dualdisk}.

If we allow even more alternations in Definiton \ref{def:aba}, Ackerman, Keszegh and Pálvölgyi \cite{abab} showed by realizing the iterated $m$-ary tree construction (Section \ref{sec:ittree}) that for every $c\ge 2$ and $m\ge 2$ there exists an ABABA-free $m$-uniform hypergraph which is not $c$-colorable.



\section{Further results and open problems}\label{sec:further}
In this section we collect the problems that appeared in this survey and highlight some more. Probably the most interesting question is the following conjecture. 

\begin{conj}\label{conj:mkx}
	If $m_2(\F)<\infty$, then $m_k(\F)<\infty$ for every $k$ for any hereditary family $\F$.
\end{conj}

Even the following stronger version might be true.

\begin{conj}
	If $m_2(\F)<\infty$, then $m_k(\F)=O(k)$ for every $k$ for any hereditary family $\F$.
\end{conj}

The most annoying special case, which received a lot of attention \cite{blessnew,colorful2,K11}, is when we want to decompose a covering by bottomless rectangles.

\begin{conj}
    $m_k^*(\F)=O(k)$ where $\F$ is the family of bottomless rectangles in the plane.
\end{conj}

For some specific geometric hypergraphs the chromatic number is still not determined.

\begin{problem}\label{prob:}
	For the translates (or homothets) of a half-disk, what is the exact value of $\chim$? 
\end{problem}

For translates, we know from Theorem \ref{thm:general_three_col} that $\chim\le 3$, while for homothets even $\chim=4$ is a possibility.

\begin{problem}\label{prob:dualdiskx}
	For the homothets of a disk is $\chi^*_{\rm big}=3$ or $4$? 
\end{problem}

\begin{problem}\label{prob:homox}
Is it true that for the homothets 
of any convex polygon we have $\chim\le 2$? 
\end{problem}

\begin{problem}
    For balls in $\R^3$ is $\chim<\infty$?
\end{problem}

Considering the coloring of the generalized Delaunay graph of points with respect to axis-parallel rectangles, the bounds in Remark \ref{remark:ap} are far apart and therefore leave the following problem open.

\begin{problem}
What is the minimal number of colors (as a function of $n$) that is always enough to color any set of $n$ points in the plane such that axis-parallel rectangles containing two points are not monochromatic? Is it polylogarithmic or polynomial in $n$?
\end{problem}

Colorings of ABA-free hypergraphs were first considered in \cite{PP} under the name {\em special shift-chains} as a subfamily of so-called {\em shift-chains}. This notion came up while studying geometric hypergraphs, as the $m$-uniform hyperedges of a hypergraph defined by the translates of an infinite plane convex set, say, an upwards parabola, always give a shift-chain.

\begin{defi}\label{def:shiftchain}
	For $A\subset [n]=\{1,2,\ldots,n\}$, denote by $a_i$ the $i$th smallest element of $A$.
	For two sets of equal sizes, $A,B\subset [n]$, we write $A\preceq B$ if $a_i\le b_i$ for every $i$.	
	An $m$-uniform hypergraph on the vertex set $[n]$ is called a {\em shift-chain} if its hyperedges are totally ordered by the relation $\preceq$.
\end{defi}

It follows from the definition that ABA-free hypergraphs are always shift-chains but the converse does not necessarily hold.
Also observe that we can get all hyperedges of a shift-chain by starting with the minimal hyperedge (with respect to the order given by $\preceq$) and increasing the elements one by one.

\begin{problem}
	Does there exist an integer $m$ such that every $m$-uniform shift-chain is $2$-colorable?
\end{problem}

Fulek \cite{F10} managed to find by an exhaustive computer search a 3-uniform shift-chain that is not 2-colorable (see Figure \ref{fig:rado}), but for larger $m$ the question is wide open, despite having received significant attention \cite{shiftchain}.
It was recently shown by Ueckerdt \cite{torstenshiftchain} that for every $m$ there is an $m$-uniform shift-chain that is not polychromatic $3$-colorable.
Note that this result has no direct relation to Conjecture \ref{conj:mk}, as shift-chains are not a hereditary family.

\begin{figure}[ht]
	\centering
	\includegraphics[width=7cm]{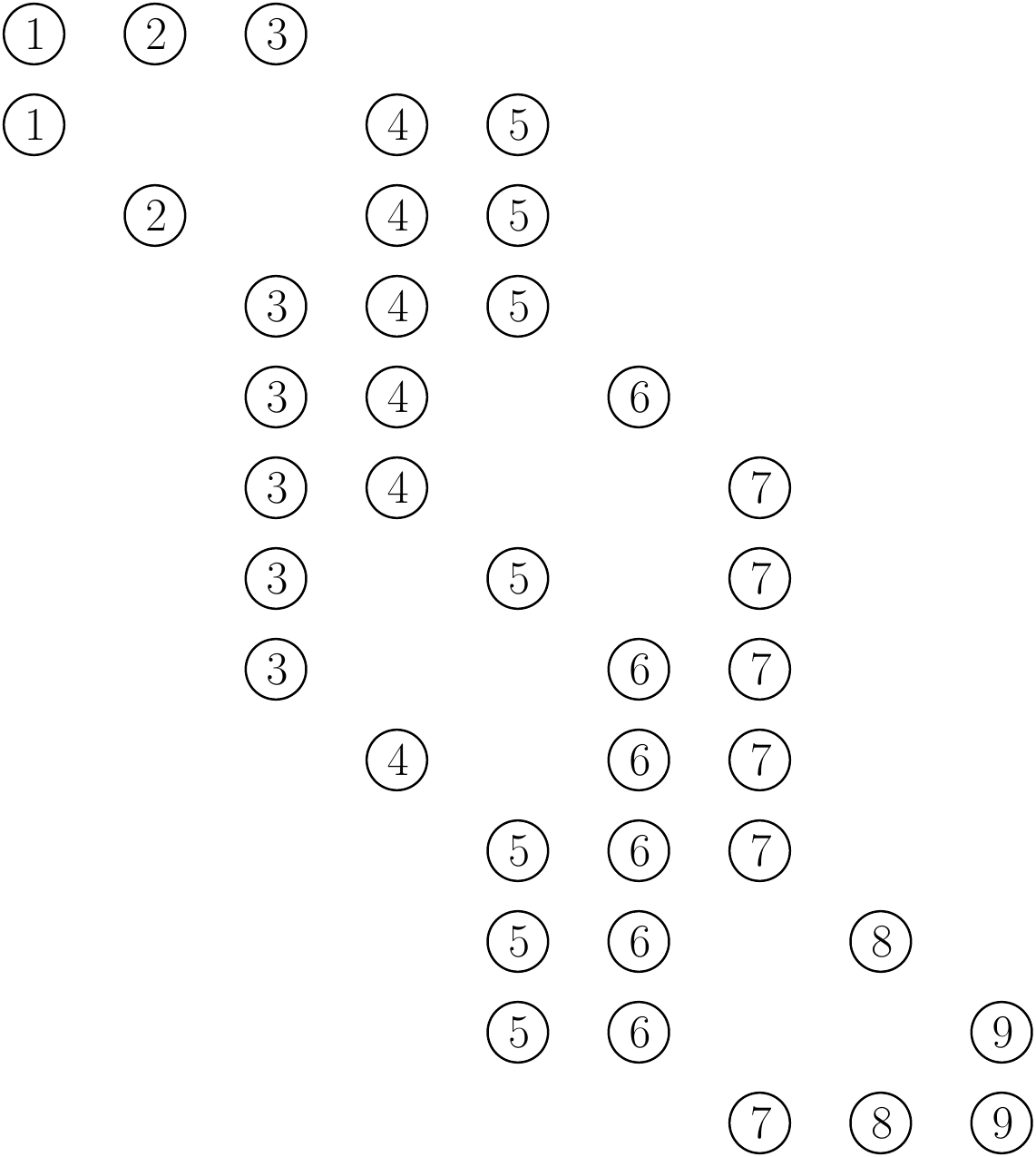} 
	\caption{A shift-chain of $13$ triples, each of which corresponds to a row.
		For any $2$-coloring of the $9$ vertices, one of the triples is monochromatic.}\label{fig:rado}
\end{figure}

\medskip

The following setup was proposed by Pálvölgyi \cite{P21}. For a set $D\subset \mathbb N$, denote by $\mathcal A_D$ the family of all finite 
arithmetic progressions of \N with difference $d\in D$.
For any $S\subset \N$, the traces of these progressions on $S$
can be interpreted as the hyperedges of a hypergraph over $S$, that is, the edges are the 
intersections of these arithmetic progressions from $\mathcal A_D$ with $S$. For a fixed $D$, we can consider the family of hypergraphs realizable this way.

\begin{problem}\label{prob:AP}
Determine or estimate $\chim$ and $m_k$ parameters of the above hypergraph families for any fixed $D$. 
\end{problem}

Note that if $D=\N$, then van der Waerden's theorem implies that $\chim=\infty$, just consider the case $S=\N$.


On the other hand, if $D$ is finite, then every vertex is contained in at most $\sum\limits_{d\in D} dm=O(m)$ hyperedges of an $m$-uniform hypergraph, making the maximum degree small enough, so that from the Lovász Local Lemma \cite{EL75} (or two-coloring each pair given by the pairing strategy of \cite{KS10, MP10}) we can conclude that $\chim=2$ and $m_k=O(k)$, and it is not hard to see that the exact value of $m_k$ depends only on the divisibility lattice of $D$. 
Pálvölgyi proposed Problem \ref{prob:AP} as it is a natural question in combinatorial number theory, and 
it is closely related  to geometric families: If $D=\{2^i\mid i\in \mathbb N\}$, then this hypergraph contains all finite hypergraphs realizable by bottomless rectangles.
A further interesting subfamily, $\mathcal A_D^\infty\subset \mathcal A_D$, consists of the arithmetic progressions with an infinite number of elements.
Recently, some partial results have been obtained about the $\chim$ parameter of these families for some $D$ by Bursics et al. \cite{bursicsek}. We also note that during the preparation of this survey the first author resolved the $D=\{2^i\mid i\in \mathbb N\}$ case, the solution will be presented in an upcoming paper.

\bigskip

\textbf{Acknowledgements.}

The authors thank Eyal Ackerman, Shakhar Smorodinsky, and Panna Geh\'er for their many useful comments.
\medskip

\textbf{Statements and Declarations.}

\textit{Funding.} Supported by the ERC Advanced Grants no.~882971 GeoScape and  no.~101054936 ERMiD, by the J\'anos Bolyai Research Scholarship of the Hungarian Academy of Sciences, 
by the 
National Research, Development and Innovation Office, NKFIH grant,
K-131529, 
the Thematic Excellence Program TKP2021-NKTA-62 and the EXCELLENCE-24 project no.~151504 Combinatorics and Geometry of the NRDI Fund.

\textit{Competing Interests.} The authors have no relevant financial or non-financial interests to disclose.

\textit{Data Availability Statement.} No data are associated with this article.
\medskip

\bibliographystyle{plainurl}
\bibliography{psdisk}

@article{AKP18,
 author = {Ackerman, Eyal and Keszegh, Bal{\'a}zs and P{\'a}lv{\"o}lgyi, D{\"o}m{\"o}t{\"o}r},
 title = {Coloring {Delaunay}-edges and their generalizations},
 fjournal = {Computational Geometry},
 journal = {Comput. Geom.},
 issn = {0925-7721},
 volume = {96},
 pages = {16 pp.},
 year = {2021},
 language = {English},
 doigda = {10.1016/j.comgeo.2021.101745},
 urlrst = {real.mtak.hu/156457/1/1806.03931.pdf},
 zbMATH = {7396283},
 Zbl = {1475.05166}
}

@misc{BCMSSS25,
      title={Polychromatic Coloring of Tuples in Hypergraphs}, 
      author={Ahmad Biniaz and Jean-Lou De Carufel and Anil Maheshwari and Michiel Smid and Shakhar Smorodinsky and Miloš Stojaković},
      eprint={2503.22449},
      archivePrefix={arXiv},
      primaryClass={cs.CG},
      urlasfd={https://arxiv.org/abs/2503.22449}, 
}

@book{nabilbook,
  title={Sampling in combinatorial and geometric set systems},
  author={Mustafa, Nabil H.},
  volume={265},
  year={2022},
  publisher={American Mathematical Society}
}

@article{bursicsek,
  title={Hitting sets and colorings of hypergraphs},
  author={Bursics, Bal{\'a}zs and Csonka, Bence and Szepessy, Luca},
  journal={arXiv preprint arXiv:2307.12154},
  year={2023}
}

@misc{torstenshiftchain,
  author       = {Torsten Ueckerdt},
  title        = {A Note on Polychromatic Colorings of Shift-Chains},
      eprint={2405.05296},
      archivePrefix={arXiv}
}

@book{bergebook,
  author    = {Berge, Claude},
  title     = {Hypergraphs},
  year      = {1989},
  publisher = {North-Holland},
  address   = {Amsterdam},
  isbn      = {0-444-87489-5},
  series    = {North-Holland Mathematical Library},
  volume    = {45}
}

@webpage{cogezoo,
	author  = {Bal{\'{a}}zs Keszegh and D{\"{o}}m{\"{o}}t{\"{o}}r P{\'{a}}lv{\"{o}}lgyi},
	title = {Geometric Hypergraph Zoo},  
	URLTemp = {http://coge.elte.hu/cogezoo.html},  
    url = {http://coge.elte.hu/cogezoo.html}  
    
}

@article{Har-PeledS05,
	author    = {Sariel Har{-}Peled and
	Shakhar Smorodinsky},
	title     = {Conflict-Free Coloring of Points and Simple Regions in the Plane},
	journal   = {Discrete {\&} Computational Geometry},
  volume    = {34},
	number    = {1},
	pages     = {47--70},
	year      = {2005},
	urlTemp       = {https://doi.org/10.1007/s00454-005-1162-6},
	doiTemp       = {10.1007/s00454-005-1162-6},
	biburlTemp    = {https://dblp.org/rec/journals/dcg/Har-PeledS05.bib},
	bibsource = {dblp computer science bibliography, https://dblp.org}
}

@Article{Kedem1986,
  author     = {Kedem, Klara and Livne, Ron and Pach, J{\'a}nos and Sharir, Micha},
  title      = {On the union of {J}ordan regions and collision-free translational motion amidst polygonal obstacles},
  journal    = {Discrete {\&} Computational Geometry},
  year       = {1986},
  volume     = {1},
  number     = {1},
  pages      = {59--71},
  issn       = {1432-0444},
  day        = {01},
  doi_jabref = {10.1007/BF02187683},
  url_jabref = {https://doi.org/10.1007/BF02187683},
}

@Article{KP15,
  author  = {Bal{\'{a}}zs Keszegh and D{\"{o}}m{\"{o}}t{\"{o}}r P{\'{a}}lv{\"{o}}lgyi},
  title   = {More on decomposing coverings by octants},
  journal = {Journal of Computational Geometry},
  year    = {2015},
  volume  = {6},
  number  = {1},
  pages   = {300--315},
}

@Article{evenlotker,
  author     = {Guy Even and Zvi Lotker and Dana Ron and Shakhar Smorodinsky},
  title      = {Conflict-Free Colorings of Simple Geometric Regions with Applications to Frequency Assignment in Cellular Networks},
  journal    = {{SIAM} Journal on Computing},
  year       = {2003},
  volume     = {33},
  number     = {1},
  pages      = {94--136},
  doi_jabref = {10.1137/s0097539702431840},
  publisher  = {Society for Industrial {\&} Applied Mathematics ({SIAM})},
}

@Article{pinchasi,
  author     = {Rom Pinchasi},
  title      = {A Finite Family of Pseudodiscs Must Include a {\textquotedblleft}Small{\textquotedblright} Pseudodisc},
  journal    = {{SIAM} Journal on Discrete Mathematics},
  year       = {2014},
  volume     = {28},
  number     = {4},
  pages      = {1930--1934},
  doi_jabref = {10.1137/130949750},
  publisher  = {Society for Industrial {\&} Applied Mathematics ({SIAM})},
}

@Article{KP11,
  author  = {Bal{\'{a}}zs Keszegh and D{\"{o}}m{\"{o}}t{\"{o}}r P{\'{a}}lv{\"{o}}lgyi},
  title   = {Octants Are Cover-Decomposable},
  journal = {Discrete {\&} Computational Geometry},
  year    = {2012},
  volume  = {47},
  number  = {3},
  pages   = {598--609},
}

@Article{AKV15,
  author     = {Eyal Ackerman and Bal{\'{a}}zs Keszegh and Máté Vizer},
  title      = {Coloring Points with Respect to Squares},
  journal    = {Discrete {\&} Computational Geometry},
  year       = {2017},
pages = {757--784},
  volume  = {58},
  doi_jabref = {10.1007/s00454-017-9902-y},
  publisher  = {Springer Nature},
}

@InProceedings{KPpropcol,
  author    = {Bal{\'{a}}zs Keszegh and D{\"{o}}m{\"{o}}t{\"{o}}r P{\'{a}}lv{\"{o}}lgyi},
  title     = {Proper Coloring of Geometric Hypergraphs},
  booktitle = {Symposium on Computational Geometry},
  year      = {2017},
  volume    = {77},
  series    = {LIPIcs},
  pages     = {47:1--47:15},
  publisher = {Schloss Dagstuhl - Leibniz-Zentrum fuer Informatik},
}

@Article{psdiskwrtpsdisk,
author={Keszegh, Bal{\'a}zs},
title={Coloring Intersection Hypergraphs of Pseudo-Disks},
journal={Discrete {\&} Computational Geometry},
year={2020},
volume={64},
number={3},
pages={942-964},
issn={1432-0444},
doiTemp={10.1007/s00454-019-00142-6},
urlTemp={https://doi.org/10.1007/s00454-019-00142-6}
}

@Article{K13,
  author    = {Kov{\'a}cs, Istv{\'a}n},
  title     = {Indecomposable coverings with homothetic polygons},
  journal   = {Discrete \& Computational Geometry},
  year      = {2015},
  volume    = {53},
  number    = {4},
  pages     = {817--824},
  publisher = {Springer},
}

@Article{cardinalkorman,
  author    = {Cardinal, Jean and Korman, Matias},
  title     = {Coloring planar homothets and three-dimensional hypergraphs},
  journal   = {Computational Geometry},
  year      = {2013},
  volume    = {46},
  number    = {9},
  pages     = {1027--1035},
  publisher = {Elsevier},
}

@Article{smorod-onthe,
  author    = {Smorodinsky, Shakhar},
  title     = {On the chromatic number of geometric hypergraphs},
  journal   = {SIAM Journal on Discrete Mathematics},
  year      = {2007},
  volume    = {21},
  number    = {3},
  pages     = {676--687},
  publisher = {SIAM},
}

@Article{K11,
  author     = {Bal{\'{a}}zs Keszegh},
  title      = {Coloring half-planes and bottomless rectangles},
  journal    = {Computational geometry},
  year       = {2012},
  volume     = {45},
  number     = {9},
  pages      = {495--507},
  bibsource  = {dblp computer science bibliography, http://dblp.org},
  biburlTemp     = {http://dblp.org/rec/bib/journals/comgeo/Keszegh12},
  doi_jabref = {10.1016/j.comgeo.2011.09.004},
}

@article{PP,
title = {Unsplittable coverings in the plane},
journal = {Advances in Mathematics},
volume = {302},
pages = {433-457},
year = {2016},
issn = {0001-8708},
doiTemp = {https://doi.org/10.1016/j.aim.2016.07.011},
URLTemp = {https://www.sciencedirect.com/science/article/pii/S0001870816309082},
author = {János Pach and Dömötör Pálvölgyi}
}

@article{KP14,
  title={{An abstract approach to polychromatic coloring: shallow hitting sets in ABA-free hypergraphs and pseudohalfplanes}},
  author={Bal{\'a}zs Keszegh and D{\"o}m{\"o}t{\"o}r P{\'a}lv{\"o}lgyi},
  journal={J. Comput. Geom.},
  year={2019},
  volume={10},
  pages={1-26}
}

@Misc{P07,
  author = {Gilles Pesant},
  title  = {personal communication},
year = {2007}
}

@Article{fekete-int,
  author        = {Sándor P. Fekete and  Phillip Keldenich},
  title         = {Conflict-Free Coloring of Intersection Graphs},
  journal       = {International Journal of Computational Geometry {\&} Applications},
  year          = {2018},  
  volume    	= {28},
  number  		= {3},
  pages         = {289--307},
}

@Article{smorod-int,
  author        = {{Keller}, Chaya and {Smorodinsky}, Shakhar},
  title         = {{Conflict-Free Coloring of Intersection Graphs of Geometric Objects}},
  journal = {Discrete {\&} Computational Geometry},
  year    = {2019},
  doi_jabref = {10.1007/s00454-019-00097-8},
  urlTemp={https://doi.org/10.1007/s00454-019-00097-8},
  pages   = {1--26},
}

@article{PTT09,
  title={Indecomposable coverings},
  author={Pach, J{\'a}nos and Tardos, G{\'a}bor and T{\'o}th, G{\'e}za},
  journal={Canadian mathematical bulletin},
  volume={52},
  number={3},
  pages={451--463},
  year={2009},
  publisher={Cambridge University Press}
}

@InCollection{surveycf,
  author        = {Smorodinsky, Shakhar},
  title         = {Conflict-free coloring and its applications},
  booktitle     = {Geometry--Intuitive, Discrete, and Convex},
  publisher     = {Springer},
  year          = {2013},
  pages         = {331--389},
  __markedentry = {[keszegh:]},
}

@article{SY,
author = "Shakhar Smorodinsky and Yelena Yuditsky",
title = "Polychromatic coloring for half-planes",
journal = "Journal of Combinatorial Theory, Series A",
volume = "119",
number = "1",
pages = "146-154",
year = "2012",
issn = "0097-3165",
}

@article{abab,
author = {Ackerman, Eyal and Keszegh, Bal{\'{a}}zs and P{\'{a}}lv{\"{o}}lgyi, D{\"{o}}m{\"{o}}t{\"{o}}r},
title = {Coloring Hypergraphs Defined by Stabbed Pseudo-Disks and {ABAB}-Free Hypergraphs},
journal = {SIAM Journal on Discrete Mathematics},
volume = {34},
number = {4},
pages = {2250-2269},
year = {2020},
}

@Article{DP20,
author={Dam{\'a}sdi, G{\'a}bor
and P{\'a}lv{\"o}lgyi, D{\"o}m{\"o}t{\"o}r},
title={Realizing an $m$-Uniform Four-Chromatic Hypergraph with Disks},
journal={Combinatorica},
year={2022},
volume={42},
number={1},
pages={1027-1048},
issn={1439-6912},
doiTemp={10.1007/s00493-021-4846-5},
urlTemp={https://doi.org/10.1007/s00493-021-4846-5}
}

@article{m2m3,
 author = {P{\'a}lv{\"o}lgyi, D{\"o}m{\"o}t{\"o}r},
 title = {Note on polychromatic coloring of hereditary hypergraph families},
 fjournal = {Graphs and Combinatorics},
 journal = {Graphs Comb.},
 issn = {0911-0119},
 volume = {40},
 number = {6},
 pages = {6},
 note = {Id/No 131},
 year = {2024},
 language = {English},
 doiadf = {10.1007/s00373-024-02836-y},
 zbMATH = {7955902},
 Zbl = {1553.05065}
}

@Article{Ber72,
author={Berge, Claude},
title={Balanced matrices},
journal={Mathematical Programming},
year={1972},
volume={2},
number={1},
pages={19-31},
issn={1436-4646},
doiTemp={10.1007/BF01584535},
urlTemp={https://doi.org/10.1007/BF01584535}
}

@article{pachtoth2009,
title = {Decomposition of multiple coverings into many parts},
journal = {Computational Geometry},
volume = {42},
number = {2},
pages = {127-133},
year = {2009},
issn = {0925-7721},
doiTemp = {https://doi.org/10.1016/j.comgeo.2008.08.002},
URLTemp = {https://www.sciencedirect.com/science/article/pii/S0925772108000746},
author = {János Pach and Géza Tóth}
}

@inproceedings{pach1980,
  title={Decomposition of multiple packing and covering},
  author={Pach, J{\'a}nos},
  booktitle={2. Kolloquium {\"u}ber Diskrete Geometrie},

  pages={169--178},
  year={1980},
  organization={Institut f{\"u}r Mathematik der Universit{\"a}t Salzburg}
}

@article{pach1986,
  title={Covering the plane with convex polygons},
  author={Pach, J{\'a}nos},
  journal={Discrete \& Computational Geometry},
  volume={1},
  pages={73--81},
  year={1986},
  publisher={Springer}
}

@article{TT07,
  title={Multiple coverings of the plane with triangles},
  author={Tardos, G{\'a}bor and T{\'o}th, G{\'e}za},
  journal   = {Discrete {\&} Computational Geometry},
  volume={38},
  number={2},
  pages={443--450},
  year={2007}
}

@article{A09,
  author    = {Aloupis,Greg and Cardinal, Jean and Collette, Sébastien and Imahori, Shinji and Korman, Matias and  Langerman, Stefan and Schwartz, Oded and  Smorodinsky, Shakhar and  Taslakian, Perouz},
  title     = {Colorful Strips},
  journal   = {Graphs and Combinatorics},
  volume    = {27},
  pages     = {327--339},
  year      = {2011}
}

@Article{A08,
author={Aloupis, Greg
and Cardinal, Jean
and Collette, S{\'e}bastien
and Langerman, Stefan
and Orden, David
and Ramos, Pedro},
title={Decomposition of Multiple Coverings into More Parts},
journal={Discrete {\&} Computational Geometry},
year={2010},
volume={44},
number={3},
pages={706-723},
issn={1432-0444},
doiTemp={10.1007/s00454-009-9238-3},
urlTemp={https://doi.org/10.1007/s00454-009-9238-3}
}

@article{PT10,
author={P{\'a}lv{\"o}lgyi, D{\"o}m{\"o}t{\"o}r and T{\'o}th, G{\'e}za},
title={Convex Polygons are Cover-Decomposable},
journal={Discrete {\&} Computational Geometry},
year={2010},
volume={43},
number={3},
pages={483-496},
issn={1432-0444},
doiTemp={10.1007/s00454-009-9133-y},
urlTemp={https://doi.org/10.1007/s00454-009-9133-y}
}

@Article{P10,
author={P{\'a}lv{\"o}lgyi, D{\"o}m{\"o}t{\"o}r},
title={Indecomposable Coverings with Concave Polygons},
journal={Discrete {\&} Computational Geometry},
year={2010},
volume={44},
number={3},
pages={577-588},
issn={1432-0444},
doiTemp={10.1007/s00454-009-9194-y},
urlTemp={https://doi.org/10.1007/s00454-009-9194-y}
}

@article{KP12,
title = {Octants are cover-decomposable into many coverings},
journal = {Computational Geometry},
volume = {47},
number = {5},
pages = {585-588},
year = {2014},
issn = {0925-7721},
doiTemp = {https://doi.org/10.1016/j.comgeo.2013.12.001},
URLTemp = {https://www.sciencedirect.com/science/article/pii/S0925772113001697},
author = {Balázs Keszegh and Dömötör Pálvölgyi}
}

@Article{KPself,
author={Keszegh, Bal{\'a}zs
and P{\'a}lv{\"o}lgyi, D{\"o}m{\"o}t{\"o}r},
title={Convex Polygons are Self-Coverable},
journal={Discrete {\&} Computational Geometry},
year={2014},
volume={51},
number={4},
pages={885-895},
issn={1432-0444},
doiTemp={10.1007/s00454-014-9582-9},
urlTemp={https://doi.org/10.1007/s00454-014-9582-9}
}

@InProceedings{A+13,
author="Asinowski, Andrei
and Cardinal, Jean
and Cohen, Nathann
and Collette, S{\'e}bastien
and Hackl, Thomas
and Hoffmann, Michael
and Knauer, Kolja
and Langerman, Stefan
and Laso{\'{n}}, Micha{\l}
and Micek, Piotr
and Rote, G{\"u}nter
and Ueckerdt, Torsten",
editor="Dehne, Frank
and Solis-Oba, Roberto
and Sack, J{\"o}rg-R{\"u}diger",
title="Coloring Hypergraphs Induced by Dynamic Point Sets and Bottomless Rectangles",
booktitle="Algorithms and Data Structures",
year="2013",
publisher="Springer Berlin Heidelberg",
address="Berlin, Heidelberg",
pages="73--84",
isbn="978-3-642-40104-6"
}

@article{blessnew,
title = {Colouring bottomless rectangles and arborescences},
journal = {Computational Geometry},
volume = {115},
pages = {102020},
year = {2023},
issn = {0925-7721},
doiTemp = {https://doi.org/10.1016/j.comgeo.2023.102020},
URLTemp = {https://www.sciencedirect.com/science/article/pii/S0925772123000408},
author = {Jean Cardinal and Kolja Knauer and Piotr Micek and Dömötör Pálvölgyi and Torsten Ueckerdt and Narmada Varadarajan}
}

@article{PacT10,
title = {Coloring axis-parallel rectangles},
journal = {Journal of Combinatorial Theory, Series A},
volume = {117},
number = {6},
pages = {776-782},
year = {2010},
issn = {0097-3165},
doiTemp = {https://doi.org/10.1016/j.jcta.2009.04.007},
URLTemp = {https://www.sciencedirect.com/science/article/pii/S0097316509001009},
author = {János Pach and Gábor Tardos}
}

@article{ap13,
title = {On coloring points with respect to rectangles},
journal = {Journal of Combinatorial Theory, Series A},
volume = {120},
number = {4},
pages = {811-815},
year = {2013},
issn = {0097-3165},
doiTemp = {https://doi.org/10.1016/j.jcta.2013.01.005},
URLTemp = {https://www.sciencedirect.com/science/article/pii/S0097316513000162},
author = {Eyal Ackerman and Rom Pinchasi}
}

@article{CPST09,
author = {Chen, Xiaomin and Pach, János and Szegedy, Mario and Tardos, Gábor},
title = {Delaunay graphs of point sets in the plane with respect to axis-parallel rectangles},
journal = {Random Structures \& Algorithms},
volume = {34},
number = {1},
pages = {11-23},
doiTemp = {https://doi.org/10.1002/rsa.20246},
URLTemp = {https://onlinelibrary.wiley.com/doi/abs/10.1002/rsa.20246},
eprintTemp = {https://onlinelibrary.wiley.com/doi/pdf/10.1002/rsa.20246},
year = {2009}
}

@inproceedings{C12,
author = {Chan, Timothy M.},
title = {Conflict-free coloring of points with respect to rectangles and approximation algorithms for discrete independent set},
year = {2012},
isbn = {9781450312998},
publisher = {Association for Computing Machinery},
address = {New York, NY, USA},
URLTemp = {https://doi.org/10.1145/2261250.2261293},
doiTemp = {10.1145/2261250.2261293},
booktitle = {Proceedings of the Twenty-Eighth Annual Symposium on Computational Geometry},
pages = {293–302},
numpages = {10},
location = {Chapel Hill, North Carolina, USA},
series = {SoCG '12}
}

@Article{ajwani,
author={Ajwani, Deepak
and Elbassioni, Khaled
and Govindarajan, Sathish
and Ray, Saurabh},
title={Conflict-Free Coloring for Rectangle Ranges Using $\tilde{O}(n^{0.382+\eps})$ Colors},
journal={Discrete {\&} Computational Geometry},
year={2012},
volume={48},
number={1},
pages={39-52},
issn={1432-0444},
doiTemp={10.1007/s00454-012-9425-5},
urlTemp={https://doi.org/10.1007/s00454-012-9425-5}
}

@misc{planken2,
      title={Shallow Hitting Edge Sets in Uniform Hypergraphs}, 
      author={Tim Planken and Torsten Ueckerdt},
      year={2023},
      eprint={2307.05757},
      archivePrefix={arXiv},
      primaryClass={math.CO}
}

@misc{F11,
      title={Coloring geometric hyper-graph defined by an arrangement of half-planes}, 
      author={Radoslav Fulek},
      eprint={1002.4529},
      archivePrefix={arXiv},
      primaryClass={math.CO}
}

@incollection {SH91,
    AUTHOR = {Snoeyink, Jack and Hershberger, John},
     TITLE = {Sweeping arrangements of curves},
 BOOKTITLE = {Discrete and computational geometry ({N}ew {B}runswick, {NJ},
              1989/1990)},
    SERIES = {DIMACS Ser. Discrete Math. Theoret. Comput. Sci.},
    VOLUME = {6},
     PAGES = {309--349},
 PUBLISHER = {Amer. Math. Soc., Providence, RI},
      YEAR = {1991},
   MRCLASS = {52B55 (68U05)},
  MRNUMBER = {1143306},
MRREVIEWER = {Ivan Stojmenovi\'{c}},
       doiTemp = {10.1090/dimacs/006/21},
       URLTemp = {https://doi.org/10.1090/dimacs/006/21},
}

@Article{matousekdiscrep,
 Author = {Matou{\v{s}}ek, Ji{\v{r}}{\'{\i}}},
 Title = {The determinant bound for discrepancy is almost tight},
 FJournal = {Proceedings of the American Mathematical Society},
 Journal = {Proc. Am. Math. Soc.},
 ISSN = {0002-9939},
 Volume = {141},
 Number = {2},
 Pages = {451--460},
 Year = {2013},
 Language = {English},
 DOIasdf = {10.1090/S0002-9939-2012-11334-6},
 zbMATH = {6141464},
 Zbl = {1259.05179}
}

@Article{linikolov,
 Author = {Li, Lily and Nikolov, Aleksandar},
 Title = {On the gap between hereditary discrepancy and the determinant lower bound},
 FJournal = {SIAM Journal on Discrete Mathematics},
 Journal = {SIAM J. Discrete Math.},
 ISSN = {0895-4801},
 Volume = {38},
 Number = {2},
 Pages = {1222--1238},
 Year = {2024},
 Language = {English},
 DOIdf = {10.1137/23M1566790},
 zbMATH = {7836005},
 Zbl = {1536.05446}
}

@Article{Raman2020,
author={Raman, Rajiv
and Ray, Saurabh},
title={Constructing Planar Support for Non-Piercing Regions},
journal={Discrete {\&} Computational Geometry},
year={2020},
volume={64},
number={3},
pages={1098-1122},
issn={1432-0444},
doiTemp={10.1007/s00454-020-00216-w},
urlTemp={https://doi.org/10.1007/s00454-020-00216-w}
}

@article{Ber08,
  author    = {Felix Bernstein},
  title     = {Zur {T}heorie der trigonometrische {R}eihen},
  journal   = {Leipz. Ber.},
  volume    = {60},
  pages     = {325--328},
  year      = {1908}
}

@article{E63,
  author    = {Paul Erd\H{o}s},
  title     = {On a combinatorial problem},
  journal   = {Nordisk Mat. Tidskr.},
  volume    = {11},
  pages     = {5--10},
  year      = {1963}
}

@Inbook{surveycd,
author="Pach, J{\'a}nos
and P{\'a}lv{\"o}lgyi, D{\"o}m{\"o}t{\"o}r
and T{\'o}th, G{\'e}za",

title="Survey on Decomposition of Multiple Coverings",
bookTitle="Geometry --- Intuitive, Discrete, and Convex: A Tribute to L{\'a}szl{\'o} Fejes T{\'o}th, 
B{\'a}r{\'a}ny, B{\"o}r{\"o}czky, Fejes T\'oth, Pach, eds",
year="2013",
publisher="Springer Berlin Heidelberg",
address="Berlin, Heidelberg",
pages="219--257",
isbn="978-3-642-41498-5",
doiTemp="10.1007/978-3-642-41498-5_9",
urlTemp="https://doi.org/10.1007/978-3-642-41498-5_9"
}

@article{FeH00,
  author    = {Uriel Feige and Magnus M. Halld\'orsson and Guy Kortsarz},
  title     = {Approximating the domatic number},
  journal    = {{SIAM} Journal on Computing},
  volume    = {32},
  pages     = {172--195},
  year      = {2002}
}

@inproceedings{B07,
  author    = {Adam L. Buchsbaum and Alon Efrat and Shaili Jain and Suresh Venkatasubramanian and Ke Yi},
  title     = {Restricted strip covering and the sensor cover problem},
  booktitle = {Proceedings of the Eighteenth Annual ACM-SIAM Symposium on Discrete Algorithms (SODA 2007)},
  pages     = {1056--1063},
  publisher = {ACM},
  year      = {2007},
  address   = {New York}
}

@book{W07,
  author    = {Peter Winkler},
  title     = {Mathematical mind-benders},
  publisher = {A K Peters},
  year      = {2007},
  address   = {Wellesley, MA},
  pages     = {137}
}

@article{W09,
  author    = {Peter Winkler},
  title     = {Puzzled: covering the plane},
  journal   = {Communications of the ACM},
  volume    = {52},
  pages     = {112},
  year      = {2009}
}

@article{GV10,
  author    = {Matt Gibson
and Kasturi Varadarajan},
  title     = {Decomposing coverings and the planar sensor cover problem},
  journal   = {Discrete {\&} Computational Geometry},
  volume    = {46},
  pages     = {313--333},
  year      = {2011}
}

@article{CFonline,
  author    = {Ke Chen and Amos Fiat and Haim Kaplan and Meital Levy and Jiri Matousek and
               Elchanan Mossel and János Pach and Micha Sharir and Shakhar Smorodinsky and Uli Wagner and Emo Welzl},
  title     = {Online Conflict-Free Coloring for Intervals},
  journal   = {SIAM J. Comput.},
  volume    = {36},
  pages     = {1342--1359},
  year      = {2006}
}

@article{DP22,
  author    = {G{\'{a}}bor Dam{\'{a}}sdi and
               D{\"{o}}m{\"{o}}t{\"{o}}r P{\'{a}}lv{\"{o}}lgyi},
  title     = {Three-chromatic geometric hypergraphs},
  journal   = {Journal of the European Mathematical Society},  
  year      = {2024}
}

@InProceedings{chekan,
author="Chekan, Vera
and Ueckerdt, Torsten",
editor="Bekos, Michael A.
and Kaufmann, Michael",
title="Polychromatic Colorings of Unions of Geometric Hypergraphs",
booktitle="Graph-Theoretic Concepts  in Computer Science",
year="2022",
publisher="Springer International Publishing",
address="Cham",
pages="144--157",
}

@article{KLP12,
  author    = {Balázs Keszegh and Nathan Lemons and Dömötör Pálvölgyi},
  title     = {Online and Quasi-online Colorings of Wedges and Intervals},
  journal   = {Order},
  volume    = {33},
  pages     = {389--409},
  year      = {2016},
  doiadsf       = {10.1007/s11083-015-9374-8},
  urafsl       = {https://doi.org/10.1007/s11083-015-9374-8}
}

@incollection{Klein,
  author    = {Rolf Klein},
  title     = {Concrete and Abstract {V}oronoi Diagrams},
  booktitle = {Lecture Notes in Computer Science},
  volume    = {400},
  publisher = {Springer},
  year      = {1989}
}

@article{colorful,
  author    = {Jean Cardinal and Kolja Knauer and Piotr Micek and Torsten Ueckerdt},
  title     = {Making Triangles Colorful},
  journal   = {Journal of Computational Geometry},
  volume    = {4},
  pages     = {240--246},
  year      = {2013}
}

@article{colorful2,
  author    = {Jean Cardinal and Kolja Knauer and Piotr Micek and Torsten Ueckerdt},
  title     = {Making Octants Colorful and Related Covering Decomposition Problems},
  journal   = {SIAM J. Discrete Math.},
  volume    = {28},
  pages     = {1948--1959},
  year      = {2014}
}

@article{KT14,
  author    = {István Kov\'acs and Géza T\'oth},
  title     = {Multiple coverings with closed polygons},
  journal   = {Electronic Journal of Combinatorics},
  volume    = {22},
  pages     = {18},
  year      = {2015}
}

@inproceedings{BBCP19,
  author    = {Ahmad Biniaz and Prosenjit Bose and Jean Cardinal and Michael Payne},
  title     = {Three-Coloring Three-Dimensional Uniform Hypergraphs},
  booktitle = {Proceedings of the 31st Canadian Conference on Computational Geometry (CCCG 2019)},
  year      = {2019}
}

@incollection{DP21,
  author    = {Gábor Damásdi and Dömötör Pálvölgyi},
  title     = {Unit disks hypergraphs are three-colorable},
  booktitle = {Extended Abstracts EuroComb 2021},
  series    = {Trends in Mathematics},
  volume    = {14},
  pages     = {483--489},
  year      = {2021},
publisher = {{Springer}},
}

@article{Levi,
  author    = {Friedrich Wilhelm Levi},
  title     = {\"{U}berdeckung eines {E}ibereiches durch {P}arallelverschiebung seines offenen {K}erns},
  journal   = {Arch. Math. (Basel)},
  volume    = {6},
  pages     = {369--370},
  year      = {1955}
}

@article{esswproof,
  author    = {Nicolas Bousquet and William Lochet and Stéphan Thomassé},
  title     = {A proof of the {E}rdős--{S}ands--{S}auer--{W}oodrow conjecture},
  journal   = {Journal of Combinatorial Theory, Series B},
  volume    = {137},
  pages     = {316--319},
  year      = {2019}
}

@incollection{shiftchain,
  author    = {Bartłomiej Bosek and Sebastian Czerwinski and Michał Dębski and Jarosław Grytczuk and Zbigniew Lonc and Piotr Rzazewski},
  title     = {Coloring Chain Hypergraphs},
  booktitle = {20 Years of the Faculty of Mathematics and Information Science},
  publisher = {Politechnika Warszawska},
  year      = {2020},
  address   = {Warszawa}
}

@misc{dualabab,
      title={On dual-{ABAB}-free and related hypergraphs}, 
      author={Balázs Keszegh and Dömötör Pálvölgyi},
      eprint={2406.13321},
      archivePrefix={arXiv},
      primaryClass={math.CO},
      urladfs={https://arxiv.org/abs/2406.13321}, 
}

@article{B11,
  author    = {Béla Bollobás and David Pritchard and Thomas Rothvoss and Alex Scott},
  title     = {Cover-decomposition and polychromatic numbers},
  journal   = {SIAM J. Discrete Math.},
  volume    = {27},
  pages     = {240--256},
  year      = {2013}
}

@misc{F10,
  author    = {Radoslav Fulek},
  title     = {Personal communication},
  year      = {2010}
}

@article{RS,
  author    = {Jaikumar Radhakrishnan and Aravind Srinivasan},
  title     = {Improved bounds and algorithms for hypergraph 2-coloring},
  journal   = {Random Structures \& Algorithms},
  volume    = {16},
  pages     = {4--32},
  year      = {2000}
}

@incollection{EL75,
  author    = {Paul Erd\H{o}s and László Lovász},
  title     = {Problems and results on 3-chromatic hypergraphs and some related questions},
  booktitle = {Infinite and Finite Sets (Colloq., Keszthely, 1973; dedicated to P. Erdős on his 60th birthday), Vol. II},
  series    = {Colloq. Math. Soc. János Bolyai},
  volume    = {10},
  pages     = {609--627},
  publisher = {North-Holland},
  address   = {Amsterdam},
  year      = {1975}
}

@article{MoserT,
author = {Moser, Robin A. and Tardos, G\'{a}bor},
title = {{A constructive proof of the general Lov\'{a}sz local lemma}},
year = {2010},
issue_date = {January 2010},
publisher = {Association for Computing Machinery},
address = {New York, NY, USA},
volume = {57},
number = {2},
issn = {0004-5411},
urlfd = {https://doi.org/10.1145/1667053.1667060},
doid = {10.1145/1667053.1667060},
journal = {Journal of the ACM},
articleno = {11},
numpages = {15}
}

@misc{P21,
  author       = {Dömötör Pálvölgyi},
  title        = {Decomposing arithmetic sequences},
  note         = {(In Hungarian)},
  year         = {2021},
  howpublished = {\url{https://coge.elte.hu/kutverseny21.pdf}}
}

@incollection{Karp1972,
  author    = {Richard M. Karp},
  title     = {Reducibility among Combinatorial Problems},
  booktitle = {Complexity of Computer Computations},
  editor    = {Raymond E. Miller and James W. Thatcher},
  publisher = {Springer},
  year      = {1972},
  pages     = {85--103},
  doifads       = {10.1007/978-1-4684-2001-2_9},
  urlfdas       = {https://doi.org/10.1007/978-1-4684-2001-2_9}
}

@book{AH89,
  title={Every planar map is four colorable},
  author={Appel, Kenneth I. and Haken, Wolfgang},
  volume={98},
  year={1989},
  publisher={American Mathematical Soc.}
}

@article{RSST97,
  title={The four-colour theorem},
  author={Robertson, Neil and Sanders, Daniel and Seymour, Paul and Thomas, Robin},
  journal={journal of combinatorial theory, Series B},
  volume={70},
  number={1},
  pages={2--44},
  year={1997},
  publisher={Elsevier}
}

@article{EM11,
  title={Graph polynomials and their applications {I}: The {T}utte polynomial},
  author={Ellis-Monaghan, Joanna A and Merino, Criel},
  journal={Structural analysis of complex networks},
  pages={219--255},
  year={2011},
  publisher={Springer}
}

@article{CRST06,
  title={The strong perfect graph theorem},
  author={Chudnovsky, Maria and Robertson, Neil and Seymour, Paul and Thomas, Robin},
  journal={Annals of mathematics},
  pages={51--229},
  year={2006},
  publisher={JSTOR}
}

@article{K75,
  title={On the computational complexity of combinatorial problems},
  author={Karp, Richard M.},
  journal={Networks},
  volume={5},
  number={1},
  pages={45--68},
  year={1975},
  
}

@article{D54,
title={Solution to advanced problem No. 4526.},
author={Descartes, Blanche},
journal={Amer. Math. Monthly},
volume={61},
pages={532},
year={1954},
}

@article{SS20,
  title={A survey of $\chi$-boundedness},
  author={Scott, Alex and Seymour, Paul},
  journal={Journal of Graph Theory},
  volume={95},
  number={3},
  pages={473--504},
  year={2020},
  publisher={Wiley Online Library}
}

@article{E59,
  title={Graph theory and probability},
  author={Erd{\H o}s, Paul},
  journal={Canadian Journal of Mathematics},
  volume={11},
  pages={34--38},
  year={1959},
  publisher={Cambridge University Press}
}

@article{K55,
  title={Aufgabe 360},
  author={Kneser, Martin},
  journal={Jahresbericht der Deutschen Mathematiker-Vereinigung},
  volume={2},
  number={27},
  pages={3--16},
  year={1955}
}

@article{L78,
  title={Kneser's conjecture, chromatic number, and homotopy},
  author={Lov{\'a}sz, L{\'a}szl{\'o}},
  journal={Journal of Combinatorial Theory, Series A},
  volume={25},
  number={3},
  pages={319--324},
  year={1978},
  publisher={Elsevier}
}

@article{B59,
  title={On the topology of the genetic fine structure},
  author={Benzer, Seymour},
  journal={Proceedings of the National Academy of Sciences},
  volume={45},
  number={11},
  pages={1607--1620},
  year={1959}
}

@inproceedings{SS01,
  title={Decidability of string graphs},
  author={Schaefer, Marcus and Stefankovic, Daniel},
  booktitle={Proceedings of the thirty-third annual ACM symposium on Theory of computing},
  pages={241--246},
  year={2001}
}

@article{FFK23,
  title={Lagerungen},
  author={T{\'o}th, L{\'a}szl{\'o} Fejes and T{\'o}th, G{\'a}bor Fejes and Kuperberg, W{\l}odzimierz},
  journal={Ebene, auf der},
  year={2023},
  publisher={Springer}
}

@article{F42,
  title={{\"U}ber die dichteste {K}ugellagerung},
  author={Fejes, L.},
  journal={Mathematische Zeitschrift},
  volume={48},
  number={1},
  pages={676--684},
  year={1942},
  publisher={Springer}
}

@article{PACH_TARDOS_2009, title={Conflict-Free Colourings of Graphs and Hypergraphs}, volume={18}, DOIds={10.1017/S0963548309990290}, number={5}, journal={Combinatorics, Probability and Computing}, author={Pach, János and Tardos, Gábor}, year={2009}, pages={819–834}}

@phdthesis{Cheilaris2008,
  author       = {Panagiotis Cheilaris},
  title        = {Conflict-Free Coloring},
  school       = {City University of New York},
  year         = {2008},
  type         = {PhD thesis},
}

@inproceedings{Erickson_2012,
  author    = {L. H. Erickson and S. M. LaValle},
  title     = {An art gallery approach to ensuring that landmarks are distinguishable},
  booktitle = {Robotics: Science and Systems VII},
  pages     = {81--88},
  year      = {2012}
}

@article{KS10,
 author = {Kruczek, Klay and Sundberg, Eric},
 title = {Potential-based strategies for tic-tac-toe on the integer lattice with numerous directions},
 fjournal = {The Electronic Journal of Combinatorics},
 journal = {Electron. J. Comb.},
 issn = {1077-8926},
 volume = {17},
 number = {1},
 pages = {research paper r5, 15},
 year = {2010},
 language = {English},
 urlTemp = {https://eudml.org/doc/230127},
 zbMATH = {5686987},
 Zbl = {1192.91058}
}

@article{MP10,
 author = {Mukkamala, Padmini and P{\'a}lv{\"o}lgyi, D{\"o}m{\"o}t{\"o}r},
 title = {Asymptotically optimal pairing strategy for tic-tac-toe with numerous directions},
 fjournal = {The Electronic Journal of Combinatorics},
 journal = {Electron. J. Comb.},
 issn = {1077-8926},
 volume = {17},
 number = {1},
 pages = {research paper n33, 6},
 year = {2010},
 language = {English},
 urlTemp = {https://eudml.org/doc/227041},
 zbMATH = {5827379},
 Zbl = {1202.91044}
}

\appendix
\begin{figure}
    \centering
    \includegraphics[width=1.0\linewidth]{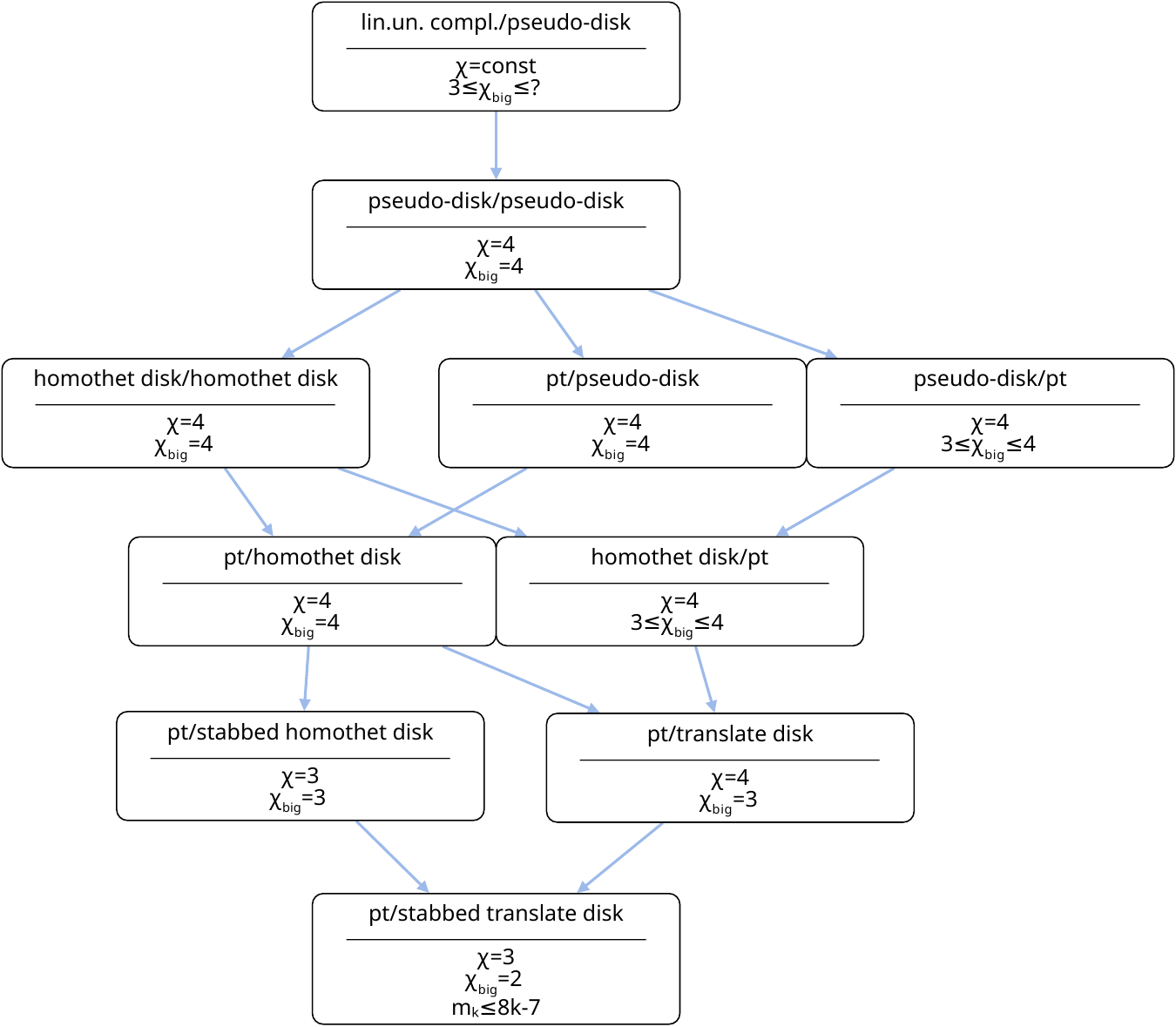}
    \caption{The inclusion hierarchy of hypergraph families defined by disks and related objects
and the best known results about their respective coloring parameters.}
    \label{fig:disks_hierarchy}
\end{figure}
\begin{figure}
    \centering
    \includegraphics[width=1.0\linewidth]{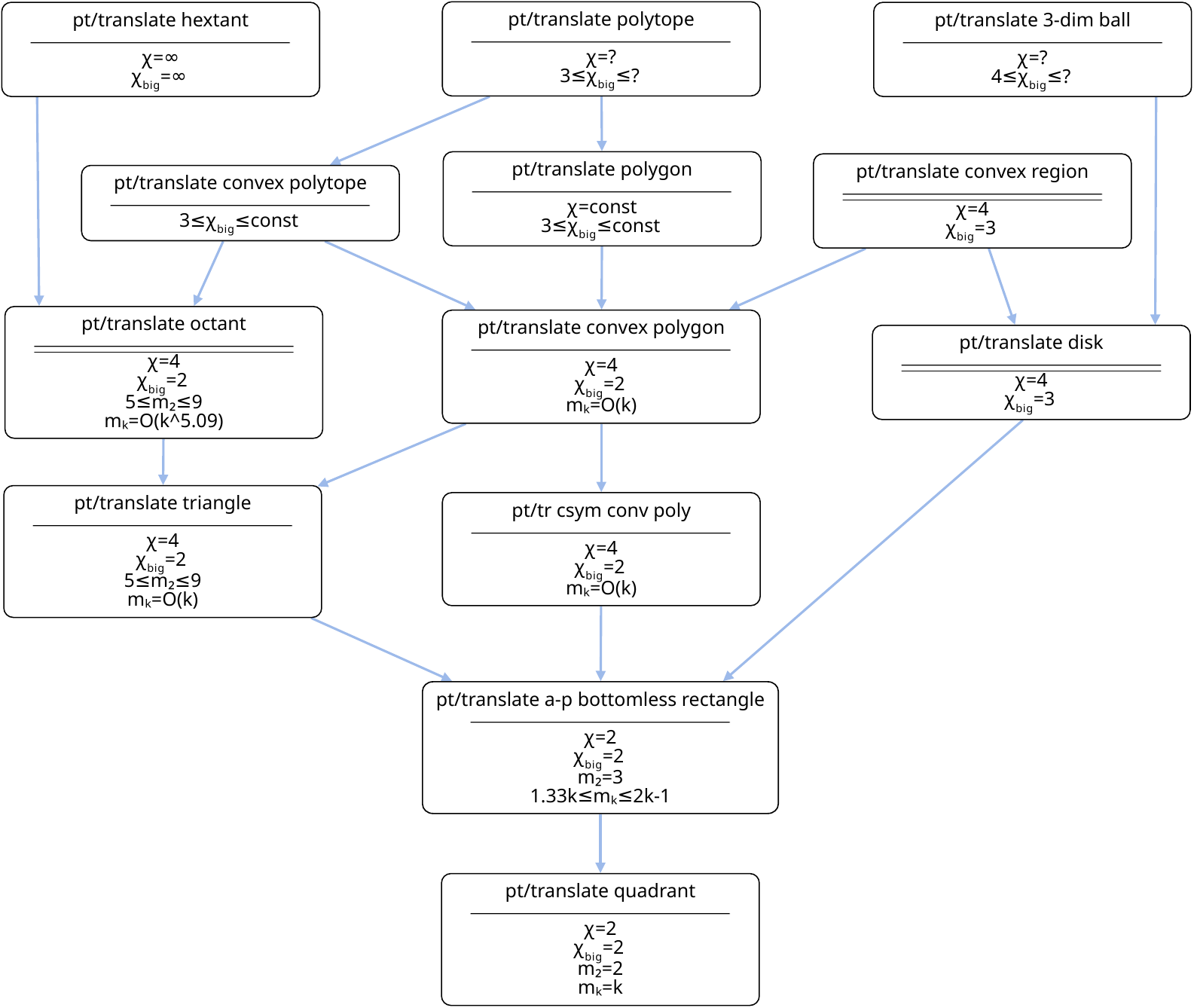}
    \caption{The inclusion hierarchy of hypergraph families defined by translates of geometric
objects and related objects
and the best known results about their respective coloring parameters.}
    \label{fig:translates_hierarchy}
\end{figure}
\begin{figure}
    \centering
    \includegraphics[width=1.0\linewidth]{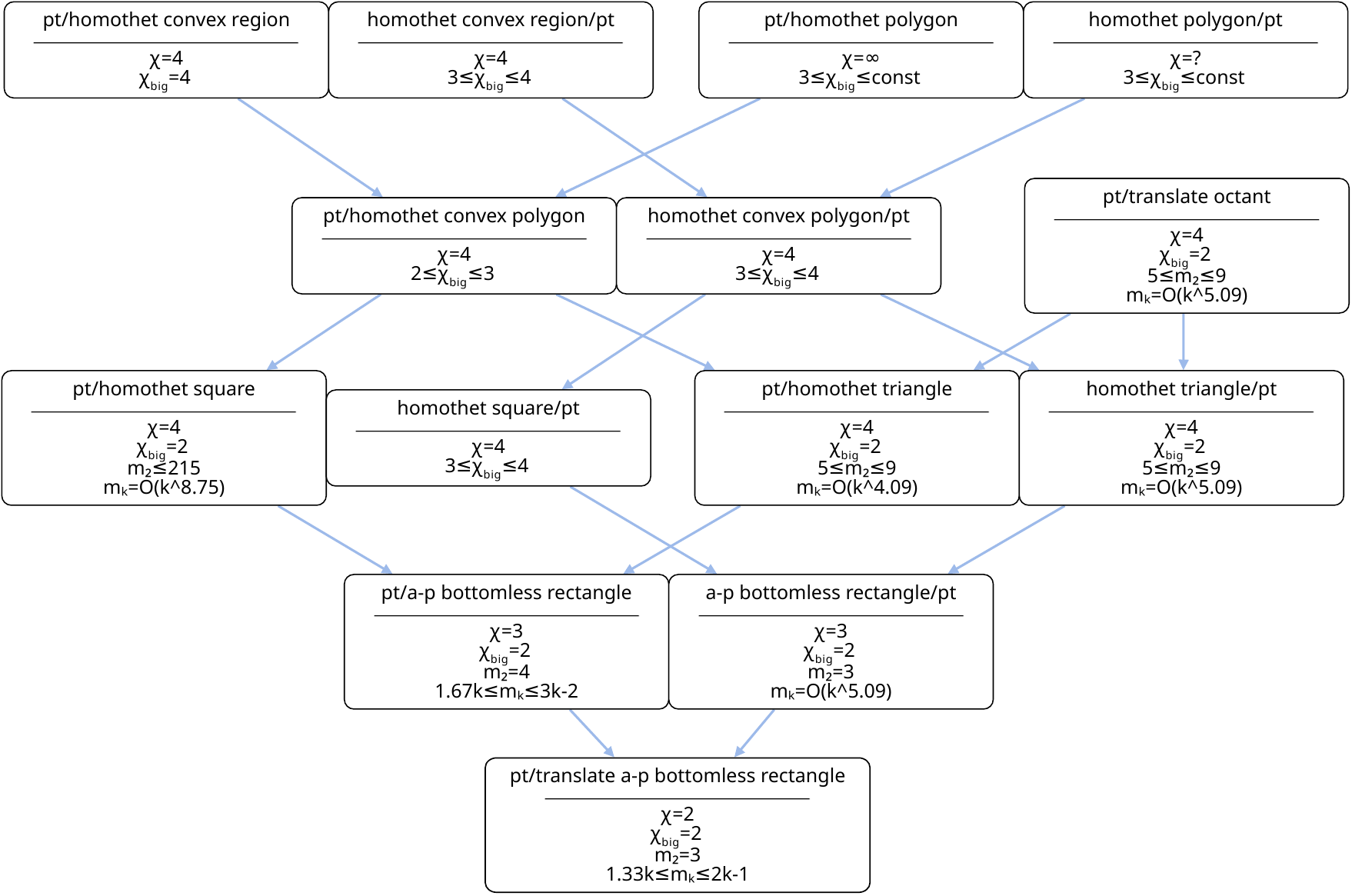}
    \caption{The inclusion hierarchy of hypergraph families defined by homothets of geometric
objects and related objects
and the best known results about their respective coloring parameters.}
    \label{fig:homothets_hierarchy}
\end{figure}
\begin{figure}
    \centering
    \includegraphics[width=1.0\linewidth]{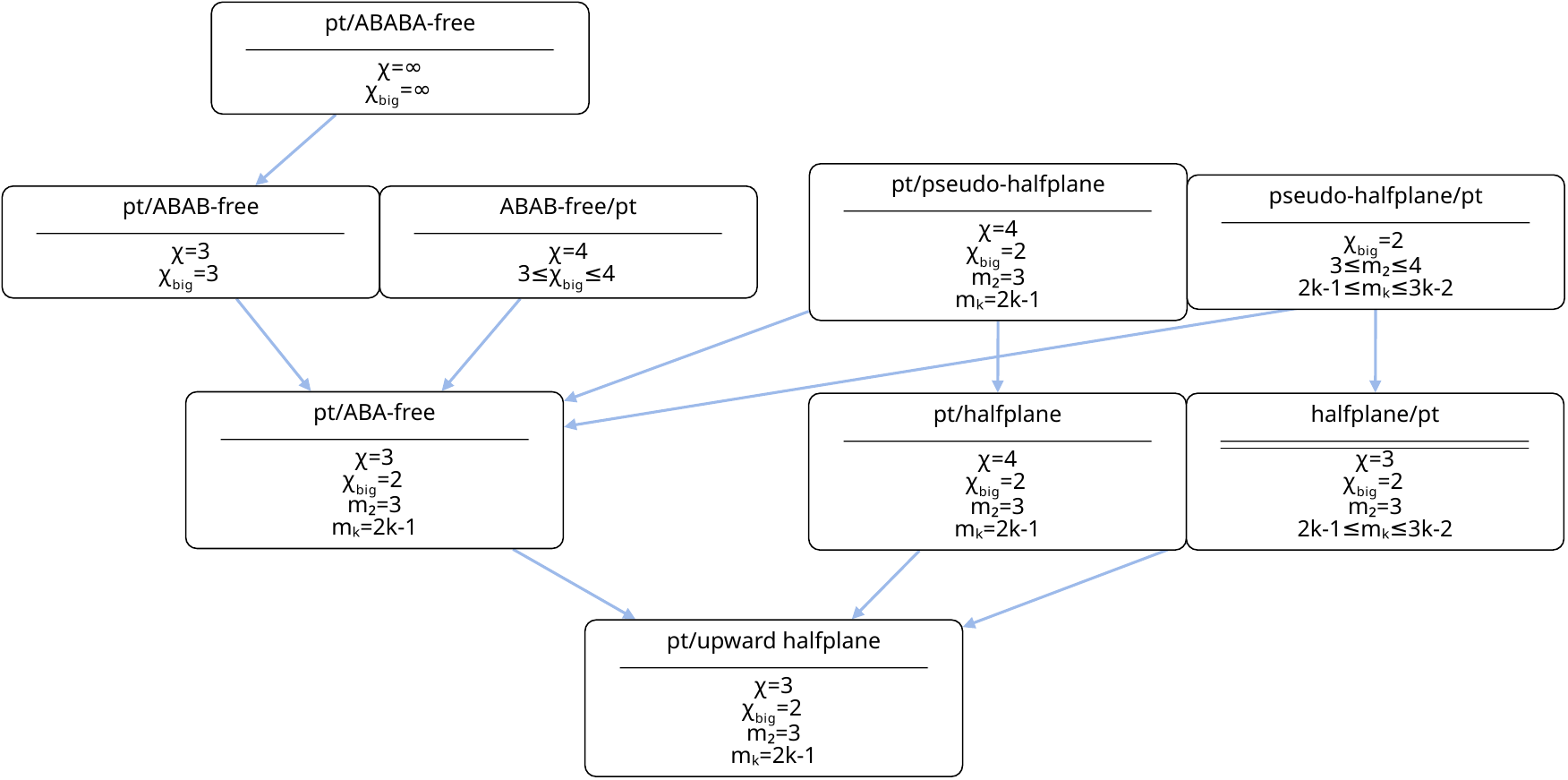}
    \caption{The inclusion hierarchy of hypergraph families defined by half-planes and related objects
and the best known results about their respective coloring parameters.}
    \label{fig:halfplanes_hierarchy}
\end{figure}

\end{document}